\documentclass[11pt]{amsart}

\usepackage{amsmath,amssymb,amsthm,mathtools}
\usepackage{xcolor}
\usepackage{tikz}
\usepackage{latexsym,enumerate,multirow}
\usepackage{nicematrix}
\usepackage{url}
\usepackage{mathrsfs}
\usepackage{ascmac}	

\allowdisplaybreaks

\usepackage[
  paperwidth=8.5in,
  paperheight=297mm,
  textwidth=6.151in,
  textheight=8.25in,
  left=1in,
  top=34mm
]{geometry}

\renewcommand{\S}{\mathbb{S}}

\newcommand{\R}{\mathbb{R}}
\newcommand{\C}{\mathbb{C}}

\newcommand{\cV}{\mathcal{V}}
\newcommand{\cD}{\mathcal{D}}

\newcommand{\cR}{\mathcal{R}}
\newcommand{\cT}{\mathcal{T}}

\newcommand{\cKD}{\mathcal{KD}}

\newcommand{\tr}{\mathop\mathrm{tr}}

\newcommand{\rank}{\mathop\mathrm{rank}}

\newcommand{\Span}{\mathop\mathrm{Span}}

\newcommand{\supp}{\mathop\mathrm{supp}}
\newcommand{\Res}{\mathop\mathrm{Res}}

\renewcommand{\epsilon}{\varepsilon}

\newtheorem{thm}{Theorem}[section]
\newtheorem{prop}[thm]{Proposition}
\newtheorem{cor}[thm]{Corollary}
\newtheorem{lemma}[thm]{Lemma}
\newtheorem{claim}[thm]{Claim}
\theoremstyle{definition}
\newtheorem{defn}[thm]{Definition}

\theoremstyle{remark}
\newtheorem{remark}[thm]{Remark}

\begin{document}

\title[Matrix Completion Discriminants]{Matrix Completion Discriminants}

\author{Yoshiyuki Sekiguchi}
\address{Graduate School of Marine Science and Technology,
Tokyo University of Marine Science and Technology,
Etchujima 2-1-8, Koto-ku, Tokyo 135-8533, Japan}
\email{yoshi-s@kaiyodai.ac.jp}

\author{Shin-ichi Tanigawa}
\address{Department of Business Economics,
School of Management, Tokyo University of Science,
1-11-2 Fujimi, Chiyoda-ku, Tokyo 102-0071, Japan}
\email{tanigawa@rs.tus.ac.jp}

\date{}



\begin{abstract}

We introduce the matrix completion discriminant of a graph as the Zariski closure of the set of partial symmetric matrices whose completion fibers fail to meet a rank stratum transversely.
A clique-sum formula for the matrix completion discriminant is established for arbitrary graphs, while explicit defining equations for cycles are obtained in terms of iterated resultants. 
Using these results together with the analysis of $K_4$-minor-free graphs, we prove that for every cycle-completable graph, the discriminant factors into determinant factors from maximal cliques and resultant factors from chordless cycles.

\end{abstract}

\maketitle


\section{Introduction}
The matrix completion problem is to fill in the missing entries of a partially
specified matrix so that the completed matrix satisfies a prescribed property.
A primary example is the low-rank matrix completion problem, in which one is asked to find a completion of low rank.
%
In this paper, we study {\em symmetric matrix} completion
and introduce a discriminant which captures degeneracies in these completion problems.
To motivate this definition, consider the problem of completing
the following symmetric matrix to one of rank at most one:
\[
\begin{pmatrix}
a_{11} & x_{12} \\ 
x_{12} & a_{22}
\end{pmatrix},
\]
where $a_{11}$ and $a_{22}$ are prescribed entries over $\mathbb{C}$
and $x_{12}$ is unknown.
For generic values of $a_{11}$ and $a_{22}$, 
there are two distinct
completions, given by $x_{12}=\pm\sqrt{a_{11}a_{22}}$.
These two solutions coincide precisely when $a_{11}a_{22}=0$,
or equivalently, when the discriminant of
$x_{12}^2-a_{11}a_{22}$ vanishes.
Geometrically, this degeneracy means that the affine space 
$\left\{
\left(\begin{smallmatrix}
a_{11} & x_{12} \\ 
x_{12} & a_{22}
\end{smallmatrix}\right)
: x_{12}\in \mathbb{C}
\right\}$ 
of completions fails to meet the rank stratum of symmetric matrices transversely if $a_{11}a_{22}=0$.

This geometric interpretation leads to our notion of discriminant,
which we now define formally.
Let $n$ be a positive integer and let $G=(V,E)$ be a graph with the vertex set 
$V=[n]=\{1,2,\dots, n\}$.
A map $a:V\cup E\rightarrow \mathbb{C}$ is called a {\em $G$-partial symmetric matrix}.
A {\em completion} of $a$ is a symmetric matrix $X$ of size $n$ such that 
$X_{ii}=a(i)$ for each $i\in V$ 
and $X_{ij}=a(ij)$ for each $ij\in E$.
Let $\mathbb{S}^n$ be the set of all symmetric matrices of size $n$ over $\mathbb{C}$,
$M_r$ be the subset of $\mathbb{S}^n$ consisting of matrices of rank $r$ for each integer $r$,
and $A_G$ be the set of all $G$-partial symmetric matrices.
The map $p_G:\mathbb{S}^n\rightarrow A_G$ is defined 
such that $X\in \mathbb{S}^n$ is a completion of $p_G(X)$.
Then, for $a\in A_G$, 
$p_G^{-1}(a)$ is the set of all completions of $a$
and $p_G^{-1}(a)\cap M_r$ is the set of all rank-$r$ completions.
A completion $X$ of $a$ is called \emph{degenerate} if
$p_G^{-1}(a)$ fails to meet $M_r$ transversely at $X$,
where $r=\rank X$.
We define the \emph{matrix completion discriminant} of $G$, or the \emph{MC discriminant}, by
$$
\cD(G)
=
\overline{
\{a\in A_G:
p_G^{-1}(a)\text{ contains a degenerate completion}\}
}^{\mathrm{Zar}}.
$$
Our definition of MC discriminants follows the notion of matrix discriminants given in
\cite[Section 2.2.3]{FruehbisKruegerZach}.
More broadly, the notion of discriminants describes the set of parameters at which a generic algebraic or geometric behavior fails,
and it arises in algebraic geometry, singularity theory, combinatorics, optimization, and algebraic statistics; see, e.g., \cite{GKZ,JiangSturmfels,SturmfelsUhler}.

The main result of this paper is a concise combinatorial description of the MC discriminant of
{\em $K_4$-minor-free graphs} or more generally {\em cycle-completable graphs}.
Our approach follows a line of research on the combinatorial aspects of the positive semidefinite matrix completion (PSD completion) problem
in combinatorial matrix theory. 
A landmark classical result in this context is a characterization of chordal graphs in terms of positive semidefinite completability 
due to Grone, Johnson, S\'a, and Wolkowicz~\cite{GroneJohnsonSaWolkowicz}.
Characterizations of $G$-partial symmetric matrices admitting PSD completions have been obtained for several basic graph classes~\cite{BarrettJohnsonTarazaga,LaurentSeriesParallel,BarrettJohnsonLoewy}.
Cycle-completable graphs mentioned above were introduced by Barrett, Johnson, and Loewy in~\cite{BarrettJohnsonLoewy} 
and play a key role in this context.

Our main result shows that, for cycle-completable graphs $G$, 
the irreducible components of the MC discriminant $\cD(G)$ are governed by maximal cliques and chordless cycles in $G$,
and a defining polynomial of $\cD(G)$ is 
\[
\prod_{L:\text{ a maximal clique}} \det X[L] \prod_{\substack{C: \text{ a chordless cycle  } \\ \text{of length at least four}}} \operatorname{ItrRes}(C)
\]
where $X=(x_{ij})$ is a symmetric matrix of indeterminates, $X[L]$ denotes its principal submatrix indexed by $L$,
and $\operatorname{ItrRes}(C)$ denotes a polynomial obtained by a simple elimination procedure along a triangulation of $C$. 
In other words, $\cD(G)$ can be computed just by enumerating maximal cliques and chordless cycles,
if $\operatorname{ItrRes}(C_m)$ has been computed beforehand for each cycle $C_m$ of length $m=4,\dots, n$.
The precise description of the main theorem will be given in Subsection~\ref{sec:main_theorem}.

As we will see in Subsection~\ref{subsec:stress}, the MC discriminant is closely related to 
the theory of {\em equilibrium (matrix-completion) stresses} in graph rigidity theory.
Indeed, Proposition~\ref{prop:stress-transversality} and Lemma~\ref{lem:1} show that a completion $X$ of $a$ is degenerate if and only if 
its minimal-rank Gram representation admits a nonzero (matrix-completion) stress.
Hence, 
$$
\cD(G)
=
\overline{
p_G( \{ X\in \mathbb{S}^n :
\text{ a minimal-rank Gram rep.~of $X$ has a nonzero stress}
\}) }^{\mathrm{Zar}}.
$$
Thus, in view of our result, one might expect the stress space of a fixed completion to be spanned by stresses supported on maximal cliques and chordless cycles when the underlying graph is cycle-completable.
This is, however, false in contrast with the PSD completion problem.
For  a PSD completion of a cycle-completable graph, 
a complete description of the set of PSD stresses is known, that is, the cone generated by 
PSD stresses  supported on maximal cliques and chordless cycles.
Without positivity, however, this fails already for $K_{2,3}$, 
as shown in Section~\ref{subsec:global-degeneracy-K23}.
To the best of our knowledge, no structural characterization of arbitrary stresses is known, even for $K_4$-minor-free graphs.
This creates a central difficulty in proving our main theorem.
Starting only from the existence of a nonzero complex stress, we must recover a clique or cycle degeneracy condition on the prescribed partial matrix, without a structural description of that stress.
Our proof addresses this difficulty.

Nondegeneracy conditions have also been studied in the context of semidefinite programs 
in \cite{AlizadehHaeberlyOverton} and play an important role in the convergence rate analysis of primal-dual interior-point methods \cite{AlizadehHaeberlyOvertonIPM}.
A directly related result to our work is a paper by Qi~\cite{QiChordalNondegeneracy}
concerning a characterization of chordal graphs in terms of constraint nondegeneracy
of positive semidefinite completions,
which can be understood as the case when $G$ is chordal and completions are real and positive semidefinite in our setting.

Laurent, Mascarin, and Telen~\cite{LaurentMascarinTelen} recently uploaded a paper concerning the algebraic boundaries of graph elliptopes and positive semidefinite completion cones.
For cycle-completable graphs, their defining polynomial for
the algebraic boundary of the PSD completion cone agrees,
up to a nonzero scalar, with that of our MC discriminant.
This agreement does not follow directly from the definition.
In fact, we will see that the 
MC discriminant may not be equal to the algebraic boundary 
if the underlying graph is not cycle-completable, as shown by the example of $K_{3,3}$.
We provide a more detailed comparison  in Section~\ref{sec:elliptopes}.

The paper is organized as follows.
Section~\ref{sec:pre} introduces the basic notions and states the main results.
Section~\ref{sec:elliptopes} compares MC discriminants with algebraic boundaries of PSD completion cones, and Sections~\ref{sec:clique}--\ref{sec:K4-atoms} develop the proof of the main theorem.
A detailed outline of the proof is given in Subsection~\ref{subsec:proof-outline}.

\section{Preliminaries and main results}
\label{sec:pre}

\subsection{Basic notions}
\label{subsec:def}

Throughout the paper, all graphs are finite and simple.
Let $G=(V,E)$ be a graph and set $n:=|V|$. 
For $S\subseteq V$, we denote by $G[S]$ the subgraph induced by $S$. 
Let $\operatorname{MaxCliq}(G)$ be the set of maximal cliques of $G$ and 
let $\operatorname{IndCyc}(G)$ be the set of induced cycles of length at least $4$. 
Suppose that $G_1$ and $G_2$ are subgraphs of $G$ such that $G=G_1\cup G_2$ and $G_1\cap G_2$ is a clique $K$, where $K$ may be the empty set.
This decomposition is called a \emph{clique-sum} and is denoted by $G=G_1+_K G_2$. It is \emph{nontrivial} if both $V(G_1)\setminus K$ and $V(G_2)\setminus K$ are nonempty. A graph admitting no nontrivial clique-sum decomposition is called a \emph{clique-sum atom}.
A graph is \emph{chordal} if it can be obtained from complete graphs by iterated clique-sums; equivalently, it has no induced cycle of length at least $4$.
A graph is \emph{$K_4$-minor-free} if $K_4$ cannot be obtained from it by deleting vertices or edges, or by contracting edges.
We call a graph \emph{cycle completable} if it can be obtained from complete graphs and $K_4$-minor-free graphs by iterated clique-sums.

Throughout the paper, 
all algebraic sets are considered over $\C$. 
For a subset $W$ of a complex affine space, let
$\overline{W}^{\mathrm{Zar}}$ and $\overline{W}^{\mathrm{Euc}}$
denote its Zariski and complex Euclidean closures, respectively.
Let $\S^V$ denote the vector space of complex symmetric matrices whose rows and columns are indexed by $V$. 
We equip $\S^V$ with the nondegenerate symmetric bilinear trace pairing $\langle Y,Z\rangle:=\tr(YZ)$.
Orthogonal complements of subspaces of $\S^V$ are taken with respect to this pairing.
When $V=[n]$, we use $\S^n$. Let $\S^V_+$ and $\S^n_+$ be the set of real positive semidefinite matrices in $\S^V$ and $\S^n$, respectively.
For $X\in\S^V$ and $I,J\subseteq V$, let $X[I,J]$ be the corresponding submatrix and set $X[I]:=X[I,I]$. 

Let $G=(V,E)$ be a graph with $n$ vertices.
Recall that $p_G:\mathbb{S}^V\rightarrow A_G$
is the projection onto the space $A_G$ of $G$-partial symmetric matrices.
We say a matrix $X\in\S^V$ is a completion of $a\in A_G$ if $p_G(X)=a$.
For a clique $L$ of $G$ and $a\in A_G$, we denote by $a[L]$ the fully specified principal submatrix indexed by $L$.
For a subgraph $H\subseteq G$, let $p_{G,H}:A_G\longrightarrow A_H$ denote the coordinate restriction to the diagonal and edge entries supported by $H$.
Thus, for $X\in\S^{V}$,
\[
p_{G,H}(p_G(X))=p_H(X[V(H)]).
\]
For $F\subseteq H\subseteq G$, these maps satisfy $p_{H,F}\circ p_{G,H}=p_{G,F}$; in particular, $p_{G,G}$ is the identity. 

For an ambient vertex set $\{v_1,\ldots,v_n\}$, let $X=(x_{ij})_{1\le i,j\le n}$ be a symmetric matrix of indeterminates, where $x_{ij}$ corresponds to the pair $(v_i,v_j)$ and $x_{ji}=x_{ij}$.
The polynomial ring
$
R_n:=\C[x_{ij}:1\le i\le j\le n]
$
will be the underlying ring throughout our discussion.
For a graph $G$ with $V(G)\subseteq\{v_1,\ldots,v_n\}$, let $R_G\subseteq R_n$ be the subring generated by the diagonal indeterminates $x_{ii}$ for $v_i\in V(G)$ and the edge indeterminates $x_{ij}$ for $\{v_i,v_j\}\in E(G)$.
We identify $R_G$ with the coordinate ring of $A_G$, so that $x_{ij}(a)=a_{ij}$ for $a\in A_G$.
For a collection $\mathcal F\subseteq R_G$ of polynomials, let $\cV_{A_G}(\mathcal F)$ denote its common zero set in $A_G$.
For a single polynomial $f$, we simply write $\cV_{A_G}(f)$.
When the ambient space is clear from the context, we omit the subscript and write $\cV(\mathcal F)$ or $\cV(f)$.
For $H\subseteq G$, the pullback by $p_{G,H}$ identifies $R_H$ with its natural subring in $R_G$. In particular, for $f\in R_H$,
\[
\cV_{A_G}(f)=p_{G,H}^{-1}\bigl(\cV_{A_H}(f)\bigr).
\]
If $L$ is a clique of $G$, then $\det X[L]\in R_G$, and its value at $a\in A_G$ is $\det a[L]$.
For an edge $e=\{i,j\}$, we write $X[e]:=X[\{i,j\}]$, and similarly for $a[e]$.
%

A {\em defining polynomial} of a hypersurface is a square-free polynomial whose zero set is the hypersurface. It is determined up to a nonzero scalar.
For polynomials $f$ and $g$ in the same polynomial ring over $\C$, we write $f\doteq g$ if $f=cg$ for some $c\in\C\setminus\{0\}$.

\subsection{Degeneracy and MC discriminants}

Recall that, for $0\le r\le n$, $M_r$ denotes the rank-$r$ stratum of $\S^V$. 
For $X\in M_r$, let $T_{M_r}(X)$ denote its tangent space at $X$. 
If $X$ is a completion of $a\in A_G$ and $r=\rank X$, we call $X$ \emph{nondegenerate} if
$$
\ker p_G+T_{M_r}(X)=\S^V,
$$
and \emph{degenerate} otherwise. 
%
An explicit description of the orthogonal complement of $T_{M_r}(X)$ yields the following characterization of degeneracy.
We denote the orthogonal complement of $\ker p_G$ by
\[
\S_{G}:=\{Z\in\S^V:Z_{ij}=0\text{ if }i\ne j\text{ and }\{i,j\}\notin E\}.
\]

\begin{prop}
\label{prop:stress-transversality}
Let $a\in A_G$ and let $X$ be a completion of $a$. Then $X$ is degenerate if and only if there exists $Z\in\S_{G}\setminus\{O\}$ such that
$XZ=O$.
\end{prop}
\begin{proof}
Let $r=\rank X$. It is a standard fact that 
$T_{M_r}(X)^\perp$ is written as $T_{M_r}(X)^\perp=\{Z\in\S^V:XZ=O\}$,
see, e.g., \cite[Section~2.1]{vanDerHolstLovaszSchrijver}.
Since $(\ker p_G)^\perp=\S_{G}$, we obtain
$
\bigl(\ker p_G+T_{M_r}(X)\bigr)^\perp
=\{Z\in\S_{G}:XZ=O\}$,
which proves the assertion.
\end{proof}

For a graph $G=(V,E)$, we define
$$
\Sigma(G):=
\{a\in A_G:p_G^{-1}(a)\text{ contains a degenerate completion}\}.
$$
By Proposition~\ref{prop:stress-transversality}, we see that
$$
\Sigma(G)
=
\{a\in A_G:
\exists\,X\in\S^V,\,Z\in\S_{G}\setminus\{O\},\ p_G(X)=a\text{ and }XZ=O\}.
$$
Then, the MC discriminant $\cD(G)$ of $G$ introduced in the introduction, which is our main target of study,  is given by
\begin{equation}\label{eq:Z}
  \begin{split}
\cD(G):&=\overline{\Sigma(G)}^{\mathrm{Zar}} \\
&=\overline{\{a\in A_G:
\exists\,X\in\S^V,\,Z\in\S_{G}\setminus\{O\},\ p_G(X)=a\text{ and }XZ=O\}}^{\mathrm{Zar}}. 
  \end{split}
\end{equation}

\subsection{Degeneracy and stresses}
\label{subsec:stress}
The formulation of degeneracy given in Proposition~\ref{prop:stress-transversality}
leads to a further connection of degeneracy to rigidity theory.
The key is the following decomposition lemma.

\begin{lemma}
\label{lem:1}
Let $X\in\S^V$ have rank $r$. There exists $U=(u_i)_{i\in V}\in\C^{r\times n}$ of rank $r$ such that $X=U^TU$. Moreover, for every $Z\in\S^V$,
$XZ=O$ if and only if $UZ=O$.
\end{lemma}
\begin{proof}
This follows from the Autonne--Takagi factorization for complex symmetric matrices \cite[Corollary~4.4.4]{HornJohnson} and the linear independence of the columns of $U^T$.
\end{proof}

We call such a factorization $X=U^TU$ a \emph{minimal-rank Gram representation} of $X$. 
Note that the Gram representation uses the symmetric bilinear form
$x^Ty$, rather than the Hermitian form $x^*y$.

For $Z\in\S_{G}$, writing $U=(u_i)_{i\in V}$, the condition $UZ=O$ is equivalent to
\[
Z_{ii}u_i+\sum_{j:\{i,j\}\in E}Z_{ij}u_j=0
\quad (i\in V).
\]
This is a matrix completion analogue of
the \emph{equilibrium conditions} for stresses in rigidity theory,
and this variant has also been studied extensively; see, e.g.~\cite{SingerCucuringu,JacksonJordanTanigawa,LaurentVarvitsiotisSAP}.
This motivates us to call $Z\in \S_{G}$ a \emph{$U$-stress} if $UZ=O$.
Similarly, for $X\in\S^V$, a matrix $Z\in\S_{G}$ is called an
\emph{$X$-stress} (or simply a \emph{stress} when $X$ is clear) if $XZ=O$.



\subsection{Small examples}
\label{subsec:small-examples}

Let $C_4$ be the cycle graph with vertices $\{1,2,3,4\}$ and $C_4^+$ be the graph obtained by adding the edge $\{1,3\}$ to $C_4$.
For a partial symmetric matrix $a$, the completions have the forms\par
{\scriptsize%
\[
X_{C_4^+}=
\begin{pmatrix}
a_{11} & a_{12} & a_{13} & a_{14}\\
a_{12} & a_{22} & a_{23} & \boldsymbol{x_{24}}\\
a_{13} & a_{23} & a_{33} & a_{34}\\
a_{14} & \boldsymbol{x_{24}} & a_{34} & a_{44}
\end{pmatrix},\quad
X_{C_4}=
\begin{pmatrix}
a_{11} & a_{12} & \boldsymbol{x_{13}} & a_{14}\\
a_{12} & a_{22} & a_{23} & \boldsymbol{x_{24}}\\
\boldsymbol{x_{13}} & a_{23} & a_{33} & a_{34}\\
a_{14} & \boldsymbol{x_{24}} & a_{34} & a_{44}
\end{pmatrix}
\]}%
for $C_4^+$ and $C_4$, respectively.
On the other hand, a stress $Z$ for degeneracy given in Proposition~\ref{prop:stress-transversality} has the forms\par
{\scriptsize%
\[
Z_{C_4^+}=
\begin{pmatrix}
z_{11} & z_{12} & z_{13} & z_{14}\\
z_{12} & z_{22} & z_{23} & 0\\
z_{13} & z_{23} & z_{33} & z_{34}\\
z_{14} & 0 & z_{34} & z_{44}
\end{pmatrix},\quad
Z_{C_4}=
\begin{pmatrix}
z_{11} & z_{12} & 0 & z_{14}\\
z_{12} & z_{22} & z_{23} & 0\\
0 & z_{23} & z_{33} & z_{34}\\
z_{14} & 0 & z_{34} & z_{44}
\end{pmatrix}
\]
}
for $C_4^+$ and $C_4$, respectively.
Equation~\eqref{eq:Z} allows us to compute $\cD(G)$ by saturation and elimination; see, e.g.,\cite[Chapter 8]{CLOIdeals}. 
Then Macaulay2 gives
\[
\cD(C_4^+)=\cV(\det X[123]\cdot \det X[134]),
\]
meaning that the two triangle determinants define the two irreducible components of $\cD(C_4^+)$.
On the other hand, a similar computation for $C_4$ gives 
\[
\cD(C_4)=\cV\left(\left(\prod_{e\in E(C_4)}
\det X[e]\right)
\cdot
\operatorname{Res}_{x_{13}}
\bigl(
\det X[123],
\det X[134]
\bigr)\right), 
\]
where $\operatorname{Res}_{x_{13}}$ 
denotes the resultant with respect to $x_{13}$.
Note that each edge forms a maximal clique in $C_4$ and hence 
the first factor $\prod_{e\in E(C_4)}\det X[e]$ is the product of the determinants associated with the maximal cliques.
The second factor is a natural polynomial analogue of 
the combinatorial fact that the cycle $1234$ can be obtained 
from the union of the two triangles $123$ and $134$ by eliminating the common edge $13$.
We will see that this phenomenon holds for all cycle-completable graphs.

\subsection{Clique-degenerate loci and essential discriminants}
\label{subsec:mc-discriminant}
In the above examples, we have $\cV(\det X[L])\subseteq\cD(G)$ for every maximal clique $L$.
We show that this is always the case for every graph $G$.
For a graph $G$, we define 
the \emph{clique-degenerate locus} by
$$
\cKD(G):=
\bigcup_{L\in\operatorname{MaxCliq}(G)}
\cV_{A_G}(\det X[L]).
$$
\begin{prop}
\label{prop:max-clique-component}
For every graph $G$, $\cKD(G)\subset\cD(G)$.
\end{prop}
\begin{proof}
Let $L\in\operatorname{MaxCliq}(G)$. Set $m:=|L|$ and $H_L:=\cV_{A_G}(\det X[L])$. 
Then $H_L$ is irreducible. 
Let $H_L^\circ$ be the set of points of $H_L$ such that $\rank a[L]=m-1$ and $\ker a[L]$ is generated by a vector $\eta=(\eta_\ell)_{\ell\in L}$ all of whose entries are nonzero. Equivalently, all diagonal cofactors of $a[L]$ are nonzero. 
$H_L^\circ$ is Zariski dense in $H_L$.

Take $a\in H_L^\circ$ and extend $\eta$ by zero to a vector $\widetilde\eta$ on $V(G)$. 
Now we construct a completion $X$ of $a$ such that $X\widetilde\eta=0$. 

First, set $X[L]=a[L]$. Then $X[L,V(G)]\widetilde\eta=a[L]\eta=0$. 
Next, let $i\notin L$. 
Since $L$ is a maximal clique, there exists $\ell_0\in L$ with $\{i,\ell_0\}\notin E(G)$. 
Since $X_{i\ell_0}$ is a free entry and $\eta_{\ell_0}\neq0$, 
we can choose the value of $X_{i\ell_0}$ so that
$\sum_{\ell\in L}X_{i\ell}\eta_\ell=0$. Doing this for every $i\notin L$ and choosing all remaining free entries arbitrarily gives a completion $X$ such that $p_G(X)=a$ and $X\widetilde\eta=X[V(G),L]\eta=0$. Let
$Z:=\widetilde\eta\widetilde\eta^T$. Since $L$ is a clique, we have $Z\in\S_{G}\setminus\{O\}$, and $XZ=O$. 

Thus $H_L^\circ\subset\Sigma(G)$. Because $H_L^\circ$ is Zariski dense in $H_L$, we obtain $H_L\subset\cD(G)$. Taking the union over all maximal cliques $L$ proves the claim.
\end{proof}

In view of Proposition~\ref{prop:max-clique-component},
the primary objective of the subsequent discussion is to understand the remaining part of $\cD(G)$.
Formally, we define the \emph{essential discriminant} of $G$ by
$$
\cR(G):=
\overline{\Sigma(G)\setminus\cKD(G)}^{\mathrm{Zar}}.
$$
 Proposition~\ref{prop:max-clique-component} implies that  $\cD(G) = \cKD(G) \cup \cR(G)$.

\subsection{Iterated resultants along a triangulation}
\label{subsec:iterated-resultants}

The example of $C_4$ in Subsection~\ref{subsec:small-examples} shows that the essential discriminant need not be empty. In this case,
\[
\cR(C_4)=\cV\!\left(
\operatorname{Res}_{x_{13}}\bigl(\det X[123],\det X[134]\bigr)
\right),
\]
and this hypersurface is not contained in $\cKD(C_4)$. 
The fact for $C_4$ extends naturally to arbitrary cycles 
by triangulating a cycle and repeatedly taking resultants of triangle determinants with respect to the chord variables. 
We now define this construction formally.

Let $n\ge3$, and let $C_n$ be the cycle on vertices $v_1,\ldots,v_n$ in cyclic order.
A \emph{triangulation} $T$ of $C_n$ is obtained by adding noncrossing chords to the cycle drawn on a circle so that every bounded face is a triangle.
For $n\ge4$, an \emph{ear} of $T$ is a triangle $\{v_s,v_i,v_t\}$ whose edges $\{v_s,v_i\}$ and $\{v_i,v_t\}$ belong to the boundary cycle $C_n$.
We call $v_i$ its \emph{ear vertex}.
Removing $v_i$ and replacing these two boundary edges by the chord $\{v_s,v_t\}$ gives a smaller cycle $C'$ with the remaining triangulation $T'$.
An \emph{ear-removal order} is a sequence of vertices removed in this way, each an ear vertex of the current triangulation, until a single triangle remains.
Pick an ear-removal order of $T$.
We define the \emph{iterated resultant} $\operatorname{ItrRes}(C_n)$ recursively as follows.
If $n=3$, set $\operatorname{ItrRes}(C_3):=\det X[V(C_3)]$.
If $n\ge4$ and $v_i$ is the first vertex in the chosen order, with ear $\{v_s,v_i,v_t\}$, set
\[
\operatorname{ItrRes}(C_n)
:=\operatorname{Res}_{x_{st}}\!\left(
\det X[\{v_s,v_i,v_t\}],\,\operatorname{ItrRes}(C')
\right),
\]
where the inner iterated resultant uses the remaining ear-removal order on $T'$.
All internal chord variables are eliminated in the end, so $\operatorname{ItrRes}(C_n)\in R_{C_n}$.
Although this definition involves a choice of triangulation $T$ and ear-removal order, the following theorem shows that $\operatorname{ItrRes}(C_n)$ is well-defined up to a nonzero scalar.
The proof is given in Section~\ref{sec:cycles}.

\begin{thm}
\label{thm:cycle-iterated-resultant}
Let $n$ be an integer with $n\ge4$.
For every triangulation $T$ of $C_n$ and every ear-removal order,
\[
\cR(C_n)=\cV\bigl(\operatorname{ItrRes}(C_n)\bigr).
\]
Moreover, $\operatorname{ItrRes}(C_n)$ is irreducible and is independent of the triangulation and the ear-removal order up to a nonzero scalar.
\end{thm}

\subsection{Main result}
\label{sec:main_theorem}
We are now ready to state our main theorem.
\begin{thm}
  \label{thm:main}
  Let $G$ be a cycle-completable graph.
  Then, $\cD(G)$ admits the following clique--cycle decomposition:
  $$
\cD(G) = \cKD(G)
\cup
\bigcup_{C\in\operatorname{IndCyc}(G)}
p_{G,C}^{-1}\cR(C).
$$
The defining polynomial for $\cD(G)$ is 
$$
\prod_{L\in\operatorname{MaxCliq}(G)}
\det X[L]
\prod_{C\in\operatorname{IndCyc}(G)}
\operatorname{ItrRes}(C),
$$
up to a nonzero scalar.
\end{thm}
Thus the defining polynomials of the irreducible components are 
determined by maximal cliques and chordless cycles.
The polynomial in Theorem~\ref{thm:main} agrees, up to a nonzero scalar, with the defining polynomial of the algebraic boundary of the positive semidefinite completion cone obtained by Laurent, Mascarin, and Telen~\cite[Corollary~4.11]{LaurentMascarinTelen}.
We discuss this connection in Section~\ref{sec:elliptopes}.

\begin{remark}
The definition of the MC discriminant also gives a decomposition by completion rank:
$
\cD(G)=\bigcup_{r=0}^{|V(G)|}\cD_r(G)$,
where $\cD_r(G)$ is the Zariski closure of the locus of partial matrices admitting a degenerate completion of rank $r$.
The individual loci $\cD_r(G)$ need not be hypersurfaces, although Theorem~\ref{thm:main} shows that their union is a hypersurface when $G$ is cycle-completable.
Moreover, a single partial matrix may admit degenerate completions of different ranks.
The decomposition by completion rank therefore need not coincide with the decomposition into irreducible components given in Theorem~\ref{thm:main}.
\end{remark}

\subsection{Outline of the proof}
\label{subsec:proof-outline}
Before turning to the proof, Section~\ref{sec:elliptopes} discusses the relation between Theorem~\ref{thm:main} and algebraic boundaries of positive semidefinite completion cones, with particular attention to  issues arising in the  setting without positive semidefiniteness.
Sections~\ref{sec:clique}--\ref{sec:K4-atoms} develop the proof of Theorem~\ref{thm:main}.

A cycle-completable graph was originally defined as a graph whose PSD completability can be determined by checking PSD completability on maximal cliques and chordless cycles.
In Subsection~\ref{subsec:def}, we have defined cycle-completable graphs using  the structural theorem of Johnson and McKee~\cite{JohnsonMcKee} (see also~\cite[Theorem~1.1]{ChudnovskyMcInnis}): 
every cycle-completable graph can be obtained by iterated clique-sums of complete graphs and $K_4$-minor-free graphs.
Accordingly, the two main tasks in the proof of the main theorem are to show that the clique--cycle decomposition is preserved under clique-sums and to establish it for $K_4$-minor-free graphs.

Section~\ref{sec:clique} addresses the first task by deriving a clique-sum formula for MC discriminants of arbitrary graphs.
As a direct application, we determine the MC discriminant of a chordal graph.
For the second task, Section~\ref{sec:cycles} first treats cycles and proves Theorem~\ref{thm:cycle-iterated-resultant}.
Section~\ref{sec:K4-atoms} first shows that the essential discriminant of each chordless cycle lifts to the MC discriminant of the whole $K_4$-minor-free graph.
It then analyzes stresses on a $K_4$-minor-free clique-sum atom to construct a rank-two completion on a suitable induced cycle.
Combining this analysis with the clique-sum formula, Subsection~\ref{sec:K4} first proves the clique--cycle decomposition for all $K_4$-minor-free graphs by induction on the number of vertices, and then extends it to cycle-completable graphs to complete the proof of Theorem~\ref{thm:main}.

\section{MC discriminants and algebraic boundaries}\label{sec:elliptopes}

We compare the MC discriminant with the algebraic boundary of the PSD completion cone.
Our main theorem (Theorem~\ref{thm:main}) implies that the two varieties coincide for cycle-completable graphs.
On the other hand, the examples of $K_{3,3}$ and $W_5$ show that this equality neither holds in general nor characterizes cycle-completable graphs.
%

We write $A_G(\R)$ and $\S_G(\R)$ for the real points of $A_G$ and $\S_G$, respectively.

\subsection{PSD completion cones and stresses}
\label{subsec:psd-boundary-stresses}

For a graph $G=(V,E)$, the \emph{PSD completion cone} is 
$
\mathscr C(G):=p_G(\S_+^V)\subseteq A_G(\R).
$
Laurent, Mascarin, and Telen~\cite{LaurentMascarinTelen} define its \emph{algebraic boundary} by
\[
\partial_a\mathscr C(G)
:=\overline{\partial\mathscr C(G)}^{\mathrm{Zar}}\subseteq A_G,
\]
where $\partial\mathscr C(G)$ is the boundary of $\mathscr C(G)$ in $A_G(\R)$ and the Zariski closure is taken over $\C$.
Since matrices in $\S_G(\R)$ have zeros in the positions corresponding to nonedges, the trace pairing induces a nondegenerate pairing between $A_G(\R)$ and $\S_G(\R)$, identifying the latter with the dual of the former.
By self-duality of the PSD cone with respect to the trace pairing, we have
 \begin{equation}
 \mathscr C(G)^*=\S_G(\R)\cap\S_+^V.
 \label{eq:psd-completion-dual}
 \end{equation}
Hence, the separation theorem leads to the following characterization of the boundary of the PSD completion cone in terms of PSD stresses:
\begin{equation}
\partial\mathscr C(G)
=\{p_G(X):X\succeq0,\ \exists\,Z\in\S_G(\R)\setminus\{O\},\ Z\succeq0,\ XZ=O\},
\label{eq:boundary-psd-stress}
\end{equation}
see, e.g.,~\cite{LaurentMascarinTelen} for more details.
Comparing this characterization with~\eqref{eq:Z} and taking Zariski closures yields
\begin{equation}
\partial\mathscr C(G)\subseteq\Sigma(G),
\quad
\partial_a\mathscr C(G)\subseteq\cD(G).
\label{eq:boundary-discriminant-inclusion}
\end{equation}
Now, the difference between the two underlying loci is explicit: 
the algebraic boundary requires both $X$ and $Z$ to be real PSD matrices, 
whereas $\Sigma(G)$ allows arbitrary complex completions and nonzero complex stresses.

\begin{remark}
The set $\Sigma(G)$ is the projection of 
$\{(X,Z)\in\S^V\times\S_G:XZ=O,\ Z\ne O\}$,
and hence is constructible by Chevalley's theorem.
Since the complex Euclidean and Zariski closures of a constructible set agree~\cite[Chapter~I, Section~10]{MumfordRedBook}, we have
$\cD(G)=\overline{\Sigma(G)}^{\mathrm{Euc}}$.
In contrast, $\partial\mathscr C(G)$ is a Euclidean closed semialgebraic cone,
and its Zariski closure can contain points in the interior of $\mathscr C(G)$ as noted in \cite{LaurentMascarinTelen}.
\end{remark}

\subsection{Comparing two varieties}
Laurent, Mascarin, and Telen~\cite{LaurentMascarinTelen} have shown\footnote{This formula combines
\cite[Corollary~4.11]{LaurentMascarinTelen}
with the resultant formula~\cite[Theorem~3.12]{LaurentMascarinTelen}
and homogenization~\cite[Proposition~3.17 and Definition~3.19]{LaurentMascarinTelen}.
These identify the homogeneous cycle polynomial factors
appearing in their algebraic boundary formula with
$\operatorname{ItrRes}(C)$, up to nonzero constants.}
that, for a cycle-completable graph $G$,
\begin{equation}
\partial_a\mathscr C(G)
=\cV_{A_G}\!\left(
\prod_{L\in\operatorname{MaxCliq}(G)}\det X[L]
\prod_{C\in\operatorname{IndCyc}(G)}\operatorname{ItrRes}(C)
\right).
\label{eq:psd-boundary-product}
\end{equation}
Combining this formula with our main theorem (Theorem~\ref{thm:main}) gives the following consequence.

\begin{cor}
\label{cor:cycle-completable-boundary}
If $G$ is cycle completable, then
$
\cD(G)=\partial_a\mathscr C(G).
$
\end{cor}

Corollary~\ref{cor:cycle-completable-boundary} follows by comparing defining polynomials 
which are obtained by independent arguments. 
This agreement naturally raises the questions of whether there is a simple, more general explanation for the coincidence of the two varieties, and how far this equality extends. 
The following proposition shows that the equality fails in general.
\begin{prop}
\label{prop:K33-MC-vs-boundary}
Let $G=K_{3,3}$ have bipartition $L=\{1,2,3\}$ and $R=\{4,5,6\}$, and let $B=X[L,R]$ be the $3\times3$ block of 
the symmetric indeterminate matrix $X$.
Then
\[
\cV_{A_G}(\det B)\subseteq\cD(G),
\quad
\cV_{A_G}(\det B)\not\subseteq\partial_a\mathscr C(G).
\]
In particular,
$\partial_a\mathscr C(K_{3,3})\subsetneq\cD(K_{3,3})$.
\end{prop}
\begin{proof}
Let $H:=\cV_{A_G}(\det B)$. Then, $H$ is an irreducible hypersurface.
For $a\in H$, $B(a)$ denotes the corresponding $3\times3$ block obtained from 
$B$ by substitution. 
If $a$ is generic in $H$, then $B(a)$ has rank $2$
and a nonzero vector $\alpha$ (resp.~$\beta$) in the left (resp.~right) kernel
has all coordinates nonzero.
We construct a $3\times 3$ symmetric matrix $P=(p_{ij})$ 
 such that $p_{ii}=a_{ii}$ for $i=1,2,3$ and $P\alpha=0$.
Indeed, the condition is described as a linear system in 
its three unspecified off-diagonal entries $p_{ij}$ 
and the corresponding matrix 
$
\left(
\begin{smallmatrix}
\alpha_2&\alpha_3&0\\
\alpha_1&0&\alpha_3\\
0&\alpha_1&\alpha_2
\end{smallmatrix}
\right)$
has the determinant $-2\alpha_1\alpha_2\alpha_3\ne0$.
Thus, such a matrix $P$ exists.
Similarly, there is a $3\times 3$ symmetric matrix $Q=(q_{ij})$ 
such that $q_{ii}=a_{i+3,i+3}$ for $i=1,2,3$ and $Q\beta=0$.
The completion and stress\par
{\scriptsize%
\[
X=\begin{pmatrix}P&B(a)\\B(a)^T&Q\end{pmatrix},
\quad
Z=\begin{pmatrix}O&\alpha\beta^T\\\beta\alpha^T&O\end{pmatrix}
\]}%
then satisfy $XZ=O$, with $Z\in\S_G\setminus\{O\}$.
Since $a$ is an arbitrary generic point in $H$, 
a dense open subset of $H$ lies in $\Sigma(G)$, and $H\subseteq\cD(G)$ follows.

Next, suppose that $H\subseteq\partial_a\mathscr C(G)$.
The algebraic boundary of a full-dimensional proper convex semialgebraic set is a hypersurface~\cite[Proposition~4.7]{LaurentMascarinTelen}, 
so $H$ would be one of the irreducible components of $\partial_a\mathscr C(G)$,
and hence the semialgebraic set $H\cap \partial\mathscr C(G)$ is Zariski dense in this irreducible component.
In particular, $H$ and $H\cap \partial\mathscr C(G)$ have the same dimension,
and accordingly there exists a nonempty relatively Euclidean open subset $O$ of $H(\R)$ contained in $H\cap \partial\mathscr C(G)$.

We choose any $a \in O$.
By~\eqref{eq:boundary-psd-stress}, 
there is a nonzero PSD stress $Z$ which is a functional supporting $\mathscr C(G)$ at $a$.
On the other hand, since $\det B$ is independent of the diagonal entries,
any small perturbation of the diagonal entries of $a$ remains in $O\subseteq \partial\mathscr C(G)$. 
For each $i\in V$, let $\delta_i\in A_G(\R)$ have a sufficiently small positive entry in the $i$th diagonal position and zeros elsewhere.
Then $a\pm\delta_i\in O\subseteq \partial\mathscr C(G)$, and the supporting property of $Z$ gives
$
0\le\langle a\pm\delta_i,Z\rangle
=\pm\langle\delta_i,Z\rangle$.
Thus $\langle\delta_i,Z\rangle=0$, so $Z_{ii}=0$ for every $i\in V$.
The positive semidefiniteness of $Z$ then implies $Z=O$. 
This is a contradiction.
\end{proof}

\begin{remark}
\label{rem:K33-Bolker-Roth-Dixon}
The construction in Proposition~\ref{prop:K33-MC-vs-boundary} is motivated by 
a classical fact in rigidity theory.
Consider a Gram representation $X=U^TU$ with $U$ of full row rank and write $U=(U_L,U_R)$.
Then by the above proof  $UZ=O$ for $Z$  implies 
$U_L\alpha=U_R\beta=0$, so the vector configuration in each part is projectively collinear.
The block $\alpha\beta^T$ is the homogeneous-coordinate version of the tensor-product stress 
constructed from affine dependencies in Bolker--Roth~\cite[Theorem~1]{BolkerRoth}.
This is the same tensor-product stress construction that occurs in 
Dixon I mechanisms of $K_{3,3}$, and it applies in arbitrary dimension.
\end{remark}


\subsection{Coincidence beyond cycle-completable graphs}
\label{subsec:non-cycle-completable-counterexamples}
The following example shows that the MC discriminant and the algebraic boundary can coincide even for graphs that are not cycle completable.
Let $W_5$ denote the wheel graph with rim vertices $1,2,3,4$ in cyclic order and hub vertex $0$.
This graph is a clique-sum atom and is not cycle completable. Consider the following completion and stress:\par
{\scriptsize %
\[
X_a
=
\begin{pmatrix}
1 & 0 & 0 & 1 & 0\\
0 & 1 & 0 & x_{13} & 1\\
0 & 0 & 1 & 1 & x_{24}\\
1 & x_{13} & 1 & 3 & 0\\
0 & 1 & x_{24} & 0 & 2
\end{pmatrix}
,\quad
 X=
 \begin{pmatrix}
  1&0&0&1&0\\
  0&1&0&1&1\\
  0&0&1&1&-1\\
  1&1&1&3&0\\
  0&1&-1&0&2
 \end{pmatrix},
 \quad
 Z=
 \begin{pmatrix}
 -1&0&-2&1&-1\\
 0&2&-2&0&-2\\
 -2&-2&-2&2&0\\
 1&0&2&-1&1\\
 -1&-2&0&1&1
 \end{pmatrix}.
\]}%
We have $XZ=O$ and hence
$a:=p_{W_5}(X)\in\Sigma(W_5)\subset\cD(W_5)$.
The maximal cliques of $W_5$ are the four triangles
$012$, $023$, $034$, and $041$, and
\[
\bigl(
\det X[012],\,
\det X[023],\,
\det X[034],\,
\det X[041]
\bigr)
=(1,1,4,1).
\]
Thus $a\notin\cKD(W_5)$. 

We next consider the contribution of chordless cycles. 
The only chordless cycle of length at least $4$ in $W_5$ is
$C=C^{1234}$. We triangulate $C$ using the diagonal $24$. Keeping the diagonal and cycle-edge entries fixed at their values in $a$ and regarding $x_{24}$ as an indeterminate, we obtain
$\det X[124]=1-x_{24}^2,\
\det X[234]=4-3x_{24}^2$,
and hence
$\operatorname{Res}_{x_{24}}
\bigl(1-x_{24}^2,\,4-3x_{24}^2\bigr)=1\neq0$.
Therefore
$
a\in
\cD(W_5)\setminus
\left(
\cKD(W_5)\cup p_{W_5,C}^{-1}\cR(C)
\right)$,
and consequently the clique--cycle decomposition fails.

Indeed, elimination with Macaulay2 shows that the irreducible components of $\cD(W_5)$ are precisely the four maximal-clique determinant hypersurfaces and the hypersurface defined by
$\operatorname{Res}_{x_{24}}\left(\det X[0124],\det X[0234]\right)$,
where $X$ again denotes the symmetric matrix of indeterminates.
The latter hypersurface does not occur in the clique--cycle decomposition above.
On the other hand, 
\cite[Example~4.15]{LaurentMascarinTelen}
 showed that 
the algebraic boundary $\partial_a \mathscr C(W_5)$ has precisely the same 
irreducible components as $\cD(W_5)$.
Thus, although the clique--cycle decomposition
fails for $W_5$, its MC discriminant coincides
with the algebraic boundary of the PSD completion cone.
It remains an open problem to characterize the graphs $G$ for which $\cD(G)=\partial_a\mathscr C(G)$.

\subsection{PSD stresses and degenerate completions}
\label{subsec:global-degeneracy-K23}


The original definition~\cite{BarrettJohnsonLoewy,JohnsonMcKee} of cycle-completable graphs is in terms of 
PSD completability by clique and cycle completability conditions,
which is, in our terminology, written as
\begin{equation}
\mathscr C(G)
=\bigcap_{L\in\operatorname{MaxCliq}(G)}
\{a\in A_G(\R):a[L]\succeq0\}
\cap\bigcap_{C\in\operatorname{IndCyc}(G)}
p_{G,C}^{-1}\bigl(\mathscr C(C)\bigr).
\label{eq:cycle-completable-psd-intersection}
\end{equation}
%
For a subgraph $H$ of $G$, let $\iota_H$ denote extension by zero from matrices indexed by $V(H)$ to matrices indexed by $V$.
Duality with respect to the trace pairing turns~\eqref{eq:cycle-completable-psd-intersection} into a decomposition of all sparse PSD matrices:
\begin{equation}
\mathscr C(G)^*
=\sum_{L\in\operatorname{MaxCliq}(G)}\iota_L(\S_+^L)
+\sum_{C\in\operatorname{IndCyc}(G)}\iota_C\bigl(\mathscr C(C)^*\bigr).
\label{eq:cycle-completable-psd-dual}
\end{equation}
%
Thus, for a fixed PSD completion $X$, every PSD $X$-stress is a sum of PSD $X$-stresses supported on maximal cliques and chordless cycles.
Consequently, every extreme ray of the PSD $X$-stress cone is supported on one of such subgraphs, cf.~\cite[Proof of Theorem 1.4]{SUY} for this duality argument.

On the other hand, the following example shows that this decomposition need not hold for non-PSD stresses. 
Consider the $K_4$-minor-free graph $K_{2,3}$ (Figure~\ref{fig:K23}) with bipartition $\{1,2\}$ and $\{3,4,5\}$, and let $a=p_{K_{2,3}}(X_a)$, where\par
{\scriptsize
\[
X_a=\begin{pmatrix}
3&\boldsymbol{3}&0&0&0\\
\boldsymbol{3}&3&0&0&0\\
0&0&2&\boldsymbol{-1}&\boldsymbol{-1}\\
0&0&\boldsymbol{-1}&2&\boldsymbol{-1}\\
0&0&\boldsymbol{-1}&\boldsymbol{-1}&2
\end{pmatrix},\qquad
Z=\begin{pmatrix}
0&0&1&1&1\\
0&0&-1&-1&-1\\
1&-1&0&0&0\\
1&-1&0&0&0\\
1&-1&0&0&0
\end{pmatrix}.
\]}
The nonedge entries of $X_a$ are shown in bold.
A real Gram representation of $X_a$ is\par
{\scriptsize%
\[
X_a=U^TU,\qquad
U=\begin{pmatrix}
1&1&1&-1&0\\
-1&-1&1&0&-1\\
1&1&0&1&-1
\end{pmatrix},
\]}%
and hence  $X_a\succeq0$ has rank $3$.
A direct computation shows $X_a Z=0$, and hence $Z$ is a nonzero $X_a$-stress.
Nevertheless, $X_a$ admits no nonzero stress supported on a maximal clique or an induced cycle.
Each maximal clique is an edge, and its principal submatrix is $\left(\begin{smallmatrix} 
3 & 0 \\ 0 & 2
\end{smallmatrix}\right)
\succ0$.
For each induced $4$-cycle $C$, the matrix $X_a[V(C)]$ has rank three
and does not admit a nonzero stress (cf.~Proposition~\ref{prop:cycle-rank2}).
Thus, $Z$ cannot be written as a linear combination of $X_a$-stresses supported on maximal cliques or chordless cycles.

\begin{remark}
The completion for $K_{2,3}$ given above is constructed based on the fact that 
the planar bar-joint realization of $K_{2,3}$ shown in Figure~\ref{fig:K23} 
has a nonzero equilibrium stress.
\end{remark}
\begin{figure}[htbp]
\centering
\begin{tikzpicture}[
  vertex/.style = {circle, draw, fill=white, inner sep=1.5pt},
  point/.style  = {circle, draw, fill=black, inner sep=1.5pt}
  ]

\node[vertex] (v5) at ( 1, 1) {5};
\node[vertex] (v2) at ( 1,-1) {2};
\node[vertex] (v3) at (-1,-1) {3};
\node[vertex] (v1) at (-1, 1) {1};
\node[vertex] (v4) at ( 0, 0) {4};

\draw (v1) -- (v3) -- (v2) -- (v5) -- (v1);
\draw (v4) -- (v1);
\draw (v4) -- (v2);

\begin{scope}[shift={(3.8,-0.2)}]

\node[point,
      label={[font=\scriptsize]below:$1,2$}] (u12) at (0,0) {};

\node[point,
      label={[font=\scriptsize]above:$3$}] (u3) at (0,1.2) {};

\node[point,
      label={[font=\scriptsize]below left:$4$}] (u4) at (-1.1,-0.8) {};

\node[point,
      label={[font=\scriptsize]below right:$5$}] (u5) at (1.1,-0.8) {};

\draw[gray] (u12) -- (u3);
\draw[gray] (u12) -- (u4);
\draw[gray] (u12) -- (u5);

\end{scope}

\end{tikzpicture}

\caption{
The graph $K_{2,3}$ (left) and a planar bar-joint framework of $K_{2,3}$ that has a nonzero equilibrium stress (right).
}

\label{fig:K23}
\end{figure}
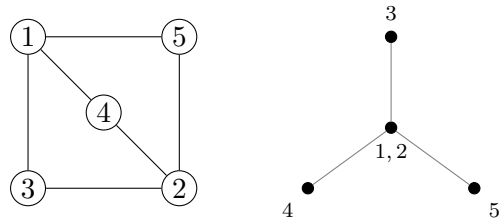

\section{Clique-sum}
\label{sec:clique}

In this section, we consider \textit{arbitrary simple graphs}. 

\subsection{Clique-sum formula}
If a given graph is a clique-sum of two graphs, then we have the following
 concise formula.

\begin{thm}[Clique-sum formula]
\label{thm:CS}
For a nontrivial clique-sum $G=G_1+_K G_2$,
\[
\cD(G)=\cKD(G)
\cup p_{G,G_1}^{-1}\cR(G_1)
\cup p_{G,G_2}^{-1}\cR(G_2).
\]
\end{thm}

Before turning to the proof, we make a brief remark.
For a clique-sum, one might expect the simpler formula
$\cD(G)=p_{G,G_1}^{-1}\cD(G_1)\cup p_{G,G_2}^{-1}\cD(G_2)$.
The following example shows that this formula fails in general and explains why the clique-degenerate terms must be treated separately.
\begin{remark}
Let $C_5^+$ be the graph obtained from $C_5$ by adding the chord $\{1,3\}$, 
with the vertices labeled cyclically as $1,2,3,4,5$.
Then we can write $C_5^+=C^{123}+_{\{1,3\}}C^{1345}$.
Recall $\cKD(G):=\bigcup_{L\in\operatorname{MaxCliq}(G)}\cV_{A_G}(\det X[L])$.
By elimination with Macaulay 2, 
we have
\begin{align*}
\cD(C_5^+) & = \cV\bigl(\det X[123]\cdot\prod_{e\in E(P^{1543})}\det X[e]
\cdot\Res_{x_{35}}(\det X[135],\det X[345])\bigr),
\intertext{where $P^{1543}$ is the path $1-5-4-3$. On the other hand,}
  \cD(C^{123}) & = \cV(\det X[123]),\\
\cD(C^{1345}) & = \cV\bigl(\prod_{e\in E(C^{1345})}\det X[e]\cdot
\Res_{x_{35}}(\det X[135],\det X[345])\bigr).
\end{align*}
Let $a\in A_{C_5^+}$ be the partial matrix whose completions have the form
{\scriptsize%
\[
X=
\begin{pmatrix}
1&0&1&x_{14}&0\\
0&2&1&x_{24}&x_{25}\\
1&1&1&0&x_{35}\\
x_{14}&x_{24}&0&1&0\\
0&x_{25}&x_{35}&0&1
\end{pmatrix}.
\]}%
Then $\det X[13]=0$, while $\det X[123]=-1$ and all the other edge factors are nonzero. Moreover,
$\Res_{x_{35}}
\bigl(\det X[135],\det X[345]\bigr)=1$.
Hence $a\notin\cD(C_5^+)$, but $p_{C_5^+,C^{1345}}(a)\in\cD(C^{1345})$.
Here, the edge $\{1,3\}$ is a maximal clique of $C^{1345}$ but is contained in the larger clique $\{1,2,3\}$ of $C_5^+$.
Its determinant is a factor of the defining polynomial of $\cD(C^{1345})$, but not of that of $\cD(C_5^+)$.
If $a[123]\succeq0$, then $\det a[13]=0$ would imply $\det a[123]=0$.
The example shows that this implication fails without positive semidefiniteness.
Thus, even for clique-sums, arguments for PSD completion that rely on this implication do not carry over directly to our setting.
\end{remark}

\subsection{Proof of the clique sum formula}
We first show the one-sided inclusion
$\cKD(G)
\cup p_{G,G_1}^{-1}\cR(G_1)
\cup p_{G,G_2}^{-1}\cR(G_2)\subset \cD(G)$.
By Proposition~\ref{prop:max-clique-component}, we already have
$\cKD(G)\subset\cD(G)$.
It remains to establish the corresponding inclusion for the essential
discriminants.
\begin{prop}
\label{prop:essential-extension}
Let $G=G_1+_K G_2$. Then
$p_{G,G_\ell}^{-1}\cR(G_\ell)\subset\cD(G)$ for $\ell=1,2$.
\end{prop}

\begin{proof}
We prove the case $\ell=1$. By definition,
$\cR(G_1)=\overline{\Sigma(G_1)\setminus\cKD(G_1)}^{\mathrm{Zar}}$.
We first show that
$p_{G,G_1}^{-1}\bigl(\Sigma(G_1)\setminus\cKD(G_1)\bigr)\subset\Sigma(G)$.
Take $a\in p_{G,G_1}^{-1}(\Sigma(G_1)\setminus\cKD(G_1))$ and set $a_1:=p_{G,G_1}(a)$. Since $a_1\in\Sigma(G_1)$, there exist a completion $X_1$ and $Z_1\in\S_{G_1}\setminus\{O\}$ such that $X_1Z_1=O$.

If $K=\varnothing$, choose any completion $X_2$ of $p_{G,G_2}(a)$ and define
a completion $X$ of $a$ by
$X[V(G_1)]=X_1,\ X[V(G_2)]=X_2,\ X[V(G_1),V(G_2)]=O$.
Extend $Z_1$ by zero to a matrix $Z$ on $G$.
Then $Z\in\S_{G}\setminus\{O\}$ and $XZ=O$, so $a\in\Sigma(G)$.

Hence we assume that $K\neq\varnothing$.
Choose a maximal clique $L$ of $G_1$ containing $K$.
Since $a_1\notin\cKD(G_1)$, we have $\det X_1[L]\neq0$, and hence $X_1[V(G_1),K]$ has full column rank. 
Here, we will construct an appropriate completion $X$ of $a$ on $G$.
Set $X[V(G_1)] = X_1$.
For each $i\in V(G_2)\setminus K = V(G)\setminus V(G_1)$, there exist $c_{ij}\in\C$ such that the row vector
\[
v^i:=\sum_{j\in V(G_1)}c_{ij}X_1[j,K]\in\C^K
\]
satisfies $v^i_j=a_{ij}$ whenever $\{i,j\}$ is an edge between $i$ and $K$. Using these $c_{ij}$, define
\[
X[i,V(G_1)]:=\sum_{j\in V(G_1)}c_{ij}X_1[j,V(G_1)],\ \forall i \in V(G_2)\setminus K.
\]
Then $X_1Z_1=O$ implies $X[i,V(G_1)]Z_1=O$. 
Since there are no edges between $V(G_1)\setminus K$ and $V(G_2)\setminus K$, such a choice of $X[i,V(G_1)\setminus K]$ does not conflict with the prescribed data $a$. Define the remaining entries of $X$
so that $X[V(G_1),V(G_2)\setminus K] = X[V(G_2)\setminus K,V(G_1)]^T$ and 
the entries of $X[V(G_2),V(G_2)]$
agree with $a$, for example, setting non-edge entries zero.
This gives a completion $X$ of $a$ on $G$. 
Extend $Z_1$ by zero to a matrix $Z$ on $G$. Then $Z\in\S_{G}\setminus\{O\}$ and $XZ=O$, so $a\in\Sigma(G)$.
Thus
$p_{G,G_1}^{-1}(\Sigma(G_1)\setminus\cKD(G_1))\subset\Sigma(G)$. 

Since $p_{G,G_1}$ is a coordinate projection and $A_G$ can be regarded as the product of $A_{G_1}$ with the remaining coordinate space, we conclude
\begin{align*}
p_{G,G_1}^{-1}\left(\cR(G_1)\right)
=p_{G,G_1}^{-1}\left(\overline{\Sigma(G_1)\setminus\cKD(G_1)}^{\mathrm{Zar}}\right)
=\overline{p_{G,G_1}^{-1}(\Sigma(G_1)\setminus\cKD(G_1))}^{\mathrm{Zar}}
\subset\cD(G).
\end{align*}
\end{proof}

Next, we show the reverse inclusion. The main step is to prove
$
\cD(G) \setminus \cKD(G)
\subset
p_{G,G_1}^{-1}\cR(G_1)
\cup p_{G,G_2}^{-1}\cR(G_2)
$.
We use the Schur complement to construct a nonzero stress on a subgraph from a stress on the whole graph.
We first note the following combinatorial fact.

\begin{lemma}
\label{lem:local-maximal-clique}
Let $G=G_1+_K G_2$ with $K\neq\varnothing$. 
If $L$ is a maximal clique of $G_1$ that is not maximal in $G$, 
then $L=K$.
\end{lemma}

\begin{proof}
This follows immediately since $K$ separates
$V(G_1)\setminus K$ from $V(G_2)\setminus K$.
\end{proof}

\begin{lemma}
\label{lem:local-stress-descent}
Let $G=G_1+_K G_2$ with $K\neq\varnothing$,
and let $\ell\in\{1,2\}$.
Suppose that $a\in A_G\setminus\cKD(G)$ has a completion $X$.
If there exists $Z_\ell\in\S_{G_\ell}\setminus\{O\}$ such that
\[
X[V(G),V(G_\ell)]Z_\ell=O,
\]
then $p_{G,G_\ell}(a)\in\cR(G_\ell)$.
\end{lemma}

\begin{proof}
We prove the case $\ell=1$. Set $X_1:=X[V(G_1)]$. Since $X_1Z_1=O$, we have
$p_{G,G_1}(a)\in\Sigma(G_1)$. If $p_{G,G_1}(a)\notin\cKD(G_1)$, the conclusion follows from the definition of $\cR(G_1)$. Thus we assume $p_{G,G_1}(a)\in\cKD(G_1)$. Then there is a maximal clique $L$ of $G_1$ such that $\det X[L]=0$. Since $a\notin\cKD(G)$, the clique $L$ is not maximal in $G$, and Lemma~\ref{lem:local-maximal-clique} gives $L=K$. In particular, $K$ is not a maximal clique in $G$.

Choose a maximal clique $K'$ of $G$ properly containing $K$, and let $J:=K'\setminus K$. 
Since $K$ is maximal in $G_1$, we have $J\subset V(G_2)\setminus K$. 
The assumption $a\notin\cKD(G)$ also gives $\det X[K']\neq0$. 
Write
\[
X[K']=
\begin{pmatrix}
X[K]&X[K,J]\\
X[J,K]&X[J]
\end{pmatrix}.
\]
We fix $\lambda\in\C$ such that both of
$\det(X[J]+\lambda I)$ and
\[
\det\begin{pmatrix}
X[K]&X[K,J]\\
X[J,K]&X[J]+\lambda I
\end{pmatrix}
\]
are nonzero. 
Let
$T:=-(X[J]+\lambda I)^{-1}$. By the Schur complement, we have
$\det\bigl(X[K]+X[K,J]\,T\,X[J,K]\bigr)\neq0$.
Define
\[
X_1(\epsilon):=X_1+\epsilon X[V(G_1),J]\,T\,X[J,V(G_1)].
\]
By the assumptions, we see that $X_1Z_1=0$ and $X[J,V(G_1)]Z_1=O$.
Thus, we obtain
$X_1(\epsilon)Z_1=O$ for every $\epsilon$, and hence
$a_1(\epsilon):=p_{G_1}(X_1(\epsilon))\in\Sigma(G_1)$. 

Now $\det X_1(\epsilon)[K]$ is not the zero polynomial in $\epsilon$ since
$\det X_1(1)[K] =\det\bigl(X[K]+X[K,J]\,T\,X[J,K]\bigr)\neq0$.
Next, let $M$ be any maximal clique of $G_1$ other than $K$.
Then $M$ is also maximal in $G$ by Lemma~\ref{lem:local-maximal-clique}. 
Thus $a\notin\cKD(G)$ implies 
$\det X_1(0)[M]\neq0$, and hence $\det X_1(\epsilon)[M]$ is also a nonzero polynomial. 
Therefore, for all but finitely many $\epsilon$, 
we have $a_1(\epsilon)\in\Sigma(G_1)\setminus\cKD(G_1)$. Since $a_1(0)=p_{G,G_1}(a)$, we conclude that $p_{G,G_1}(a)\in\cR(G_1)$. 
\end{proof}

\begin{proof}
[Proof of Theorem \ref{thm:CS}]
By Proposition~\ref{prop:max-clique-component} and Proposition~\ref{prop:essential-extension},
$
\cD(G)\supset \cKD(G)\cup p_{G,G_1}^{-1}\cR(G_1)\cup p_{G,G_2}^{-1}\cR(G_2)$.
We prove the reverse inclusion. 
Since the right-hand side is Zariski closed, it is enough to consider $a\in\Sigma(G)$. 
There is nothing to prove in the case $a\in\cKD(G)$, so assume $a\notin\cKD(G)$. 
Now there exist a completion $X$ of $a$ and a stress $Z\in\S_{G}\setminus\{O\}$ such that $XZ=O$.

If $K=\varnothing$, then
we see that $Z=
\left(
\begin{smallmatrix}
 Z_1&O\\
 O&Z_2
\end{smallmatrix}
\right)
$, 
$Z_\ell\in\S_{G_\ell}$,
and $Z_\ell\neq O$ for at least one $\ell\in\{1,2\}$.
For such an $\ell$, we have $X[V(G_\ell)]Z_\ell=O$,
and hence $p_{G,G_\ell}(a)\in\Sigma(G_\ell)$.
Moreover, every maximal clique of $G_\ell$ is also a maximal clique of $G$.
Thus $p_{G,G_\ell}(a)\notin\cKD(G_\ell)$.
Therefore, $p_{G,G_\ell}(a)\in
\Sigma(G_\ell)\setminus\cKD(G_\ell)
\subset\cR(G_\ell)$,
which proves the desired inclusion.

Assume that $K\neq\varnothing$.
Let $A:=V(G_1)\setminus K$ and $B:=V(G_2)\setminus K$, and order the vertices as $A,K,B$. 
Since there are no edges between $A$ and $B$, we can write
\[
Z=\begin{pmatrix}
Z_{AA}&Z_{AK}&0\\
Z_{KA}&Z_{KK}&Z_{KB}\\
0&Z_{BK}&Z_{BB}
\end{pmatrix},
\]
and write $X$ with the same block structure.
The $B$-column block of $XZ=O$ gives
\begin{equation}
 \begin{aligned}
X_{AK}Z_{KB}+X_{AB}Z_{BB}&=O,\\
X_{KK}Z_{KB}+X_{KB}Z_{BB}&=O,\\
X_{BK}Z_{KB}+X_{BB}Z_{BB}&=O.
\end{aligned}
\label{eq:XZ_block}
\end{equation}
We first show that $\ker Z_{BB}\subset\ker Z_{KB}$.
Suppose $\ker Z_{BB}\not\subset\ker Z_{KB}$. Then there exists $\eta\in\ker Z_{BB}$ such that $Z_{KB}\eta\neq0$. Multiplying \eqref{eq:XZ_block} on the right by $\eta$ yields
$X_{AK}Z_{KB}\eta=0$, $X_{KK}Z_{KB}\eta=0$ and $X_{BK}Z_{KB}\eta=0$. 
Thus we have $X[V(G),K]Z_{KB}\eta=0$. 
Let $L$ be a maximal clique of $G$ containing $K$. Then, in particular, 
$X[L,K]Z_{KB}\eta=0$ with $Z_{KB}\eta\neq0$. 
Thus $X[L]$ has a nonzero kernel and hence $\det X[L]=0$. This contradicts $a\notin\cKD(G)$. Therefore,
$\ker Z_{BB}\subset\ker Z_{KB}$.

Suppose that $Z_{BB}\neq0$. From $\ker Z_{BB}\subset\ker Z_{KB}$, there is a matrix $\Lambda$ such that
$Z_{KB}=\Lambda Z_{BB}$. Set
\[
Z_2:=
\begin{pmatrix}\Lambda\\ I\end{pmatrix}
Z_{BB}
\begin{pmatrix}\Lambda^T&I\end{pmatrix}
=
\begin{pmatrix}
\Lambda Z_{BB}\Lambda^T&Z_{KB}\\
Z_{BK}&Z_{BB}
\end{pmatrix}.
\]
Then \eqref{eq:XZ_block} implies 
$X[V(G),V(G_2)]Z_2 
= X[V(G),V(G_2)]
\left[
\begin{smallmatrix}
 Z_{KB} \\
 Z_{BB}
\end{smallmatrix}\right]
\left[
\begin{smallmatrix}
 \Lambda^T & I
\end{smallmatrix}
\right]
 = O$. 
In addition, since $K$ is a clique, we have $Z_2\in\S_{G_2}\setminus\{O\}$. 
Therefore, Lemma~\ref{lem:local-stress-descent} gives
$p_{G,G_2}(a)\in\cR(G_2)$.

Finally suppose that $Z_{BB}=0$. 
Since $\ker Z_{BB} \subset \ker Z_{KB}$, we have $Z_{KB} = 0$.
Thus $Z$ is the zero extension of a nonzero matrix $Z_1\in\S_{G_1}$ and
$X[V(G),V(G_1)]Z_1=O$. By Lemma~\ref{lem:local-stress-descent},
 we have $p_{G,G_1}(a)\in\cR(G_1)$. This proves that $\Sigma(G)$ is contained in the right-hand side. Taking Zariski closures, we obtain the reverse inclusion. 
\end{proof}

\subsection{MC discriminants for chordal graphs}
\label{subsec:chordal-discriminant}

As a direct application of the clique-sum formula, we give an explicit description of the MC discriminant for chordal graphs.

\begin{prop}
\label{prop:chordal-discriminant}
Let $H$ be a chordal graph. Then
$\cD(H)=\cKD(H)$.
\end{prop}

\begin{proof}
For the complete graph $K_m$, a partial matrix $a$ belongs to $\Sigma(K_m)$ if and only if $a$ is singular. 
Hence $\cD(K_m)=\cKD(K_m)$.
Every chordal graph is obtained from complete graphs by iterated clique-sums, so it remains to show that this identity is preserved under a nontrivial clique-sum $H=H_1+_K H_2$.
If $\cD(H_\ell)=\cKD(H_\ell)$ for $\ell=1,2$, then $\cR(H_\ell)=\varnothing$.
The clique-sum formula (Theorem~\ref{thm:CS}) therefore gives $\cD(H)=\cKD(H)$.
\end{proof}
\begin{remark}
A closely related characterization of chordal graphs in terms of
\textit{constraint nondegeneracy} in semidefinite programming was given in Qi~\cite{QiChordalNondegeneracy}.
We call $a\in A_G(\mathbb R)$ 
a \emph{$G$-partial positive definite} matrix if $a[L]$ is positive definite for every clique $L$ of $G$.
He proved that a graph $G$ is chordal if and only if, for every $G$-partial positive definite matrix admitting a PSD completion, constraint nondegeneracy holds at every PSD completion. In our formulation, this constraint nondegeneracy condition is equivalent to the transversality condition used in Proposition~\ref{prop:stress-transversality}.
Thus Proposition~\ref{prop:chordal-discriminant} may be viewed as a complex-algebraic counterpart of the chordal case of Qi's result. Indeed, if $H$ is chordal and a real $H$-partial matrix $a$ is positive definite on every maximal clique, then $a\notin\cKD(H)=\cD(H)$. Consequently, the fiber over $a$ contains no degenerate completion point, even over $\C$. In particular, every PSD completion is constraint nondegenerate.
\end{remark}

\section{MC discriminants for cycles}
\label{sec:cycles}

In this section, we determine the defining polynomial of the discriminant of a cycle.
For integers $n\ge3$ and $3\le m\le n$, let $C_m$ be a cycle on the vertices $v_1,\ldots,v_m$.
Throughout the section, the vertices $v_1,\ldots,v_n$ are in cyclic order, and indices are taken modulo $n$.
Let $e_i:=\{v_i,v_{i+1}\}$.
For $n\ge4$, Proposition~\ref{prop:max-clique-component} gives
\[
\cKD(C_n)=\bigcup_{i=1}^n\cV_{A_{C_n}}(\det X[e_i]),
\quad
\cD(C_n)=\cKD(C_n)\cup\cR(C_n).
\]
Thus to determine $\cD(C_n)$, it remains to determine $\cR(C_n)$.

\subsection{Essential degeneracy of cycles}
Matrix completion problems for cycles have been extensively studied and 
the following two propositions are folklore facts.
For completeness, we provide elementary proofs of both propositions.

\begin{prop}
\label{prop:cycle-rank2}
Let $n$ be an integer with $n\ge4$.
Let $a\in A_{C_n}\setminus \cKD(C_n)$ and let $X$ be a completion of $a$. 
Then the following two conditions are equivalent.
\begin{enumerate}[(1)]
\item There exists $Z\in\S_{C_n}\setminus \{O\}$ such that $XZ=O$.
\item $\rank X=2$.
\end{enumerate}
Moreover, when these conditions hold, the space of $X$-stresses is one-dimensional.
\end{prop}

\begin{proof}
First, we prove $(1)\Rightarrow(2)$. Let $X_{\bullet, i}$ be the $i$th column of $X$. Since $\det X[e_i]\neq0$, the columns $X_{\bullet, i}$ and $X_{\bullet,i+1}$ are linearly independent. Suppose that the three columns $X_{\bullet,i-1},X_{\bullet, i},X_{\bullet,i+1}$ are linearly independent for some $i$. 
Since the support of $Z_{\bullet,i}$ is contained in $\{i-1,i,i+1\}$, $XZ_{\bullet,i}=0$ gives
$z_{i-1,i}=z_{ii}=z_{i,i+1}=0$. Then the equation $XZ_{\bullet, i+1} = 0$ becomes
\[
z_{i+1,i+1}X_{\bullet,i+1}+z_{i+1,i+2}X_{\bullet,i+2}=0.
\]
By the linear independence of $X_{\bullet,i+1}$ and $X_{\bullet,i+2}$, we obtain
$z_{i+1,i+1}=z_{i+1,i+2}=0$. Repeating this argument yields $Z=O$. This is a contradiction. Hence, for every $i$, the three columns
$X_{\bullet,i-1},X_{\bullet i},X_{\bullet,i+1}$ are linearly dependent. Since every two consecutive columns are linearly independent,
\[
\operatorname{Span}\{X_{\bullet,i-1},X_{\bullet i}\}
=
\operatorname{Span}\{X_{\bullet i},X_{\bullet,i+1}\},\ \forall i \in [n].
\]
and all columns lie in the same two-dimensional subspace. Therefore $\rank X=2$.

Next, we prove $(2)\Rightarrow(1)$. By Lemma~\ref{lem:1}, we write
$X=U^TU$ with $U=(u_1,\ldots,u_n)\in\C^{2\times n}$. 
Note that
$
\det(u_i,u_{i+1})^2 = 
\left|
\begin{smallmatrix}
 \langle u_i,u_i \rangle & \langle u_i,u_{i+1} \rangle \\
 \langle u_{i+1},u_i \rangle & \langle u_{i+1},u_{i+1} \rangle
\end{smallmatrix}
\right|
= \det X[e_i] \neq 0.
$
Now we will construct an appropriate $Z$.
Set $z_i:=\det(u_i,u_{i+1})^{-1}$. Then
\[
\det\bigl(u_i,\, z_{i-1}u_{i-1}+z_i u_{i+1}\bigr)
=-z_{i-1}\det(u_{i-1},u_i)+z_i\det(u_i,u_{i+1})=0.
\]
Hence there exists $\lambda_i\in\C$ such that
$z_{i-1}u_{i-1}+z_i u_{i+1}=\lambda_i u_i$. Define
$Z_{i,i+1}=Z_{i+1,i}:=z_i$, $Z_{ii}:=-\lambda_i$, and set the other off-diagonal entries equal to $0$. 
Then $z_{i-1}u_{i-1} - \lambda_i u_i + z_i u_{i+1}=0$ gives
$UZ = O$ and hence $XZ = O$.

It remains to show uniqueness up to a scalar. Let $\widetilde Z$ be any $U$-stress and let
$\widetilde z_i:=\widetilde Z_{i,i+1}$. The $i$th equilibrium equation is
$$
\widetilde z_{i-1}u_{i-1}
+\widetilde Z_{ii}u_i
+\widetilde z_i u_{i+1}=0.
$$
Taking the determinant with $u_i$ gives
$\widetilde z_i\det(u_i,u_{i+1}) = \widetilde z_{i-1}\det(u_{i-1},u_i)$.
Thus there exists $c\in\C$ such that $\widetilde z_i=c z_i$ for every $i$. Hence $\widetilde Z-cZ$ is a $U$-stress whose off-diagonal entries all vanish. Its $i$th equilibrium equation is $(\widetilde Z_{ii}-cZ_{ii})u_i=0$. 
Since $u_i\neq0$, we have $\widetilde Z_{ii}=cZ_{ii}$ for every $i$. Therefore
$\widetilde Z=cZ$, and the space of $U$-stresses, equivalently of $X$-stresses, is one-dimensional.
\end{proof}

Proposition~\ref{prop:cycle-rank2} motivates us to look at the following map.
For every $m$, define
\[
\Omega_m:=\{U=(u_1,\ldots,u_m)\in(\C^2)^m:
\det(u_i,u_j)\neq0\text{ for all }\{v_i,v_j\}\in E(C_m)\}
\]
and
$\Phi_m:\Omega_m\longrightarrow A_{C_m},
\ U\longmapsto p_{C_m}(U^TU)$.

\begin{prop}
\label{prop:cycle-essential-hypersurface}
Let $n$ be an integer with $n\ge4$. Then
$
\cR(C_n)=\overline{\Phi_n(\Omega_n)}^{\mathrm{Zar}}$.
Moreover, $\cR(C_n)$ is an irreducible hypersurface in $A_{C_n}$.
\end{prop}

\begin{proof}
First, Proposition~\ref{prop:cycle-rank2} gives $\Sigma(C_n)\setminus \cKD(C_n) = \Phi_n(\Omega_n)$, and hence the definition of the essential discriminant yields $\cR(C_n)=\overline{\Phi_n(\Omega_n)}^{\mathrm{Zar}}$. Since $\Omega_n$ is irreducible, so is $\cR(C_n)$.
For $U\in\Omega_n$, the space of linear dependencies among the rows of the Jacobian of $\Phi_n$ at $U$ is isomorphic to the space of $U$-stresses.
(This is the matrix completion analogue of a well-known fact that the left kernel of a rigidity matrix is the space of equilibrium stresses of the underlying framework, see, e.g.~\cite[Section~3.2]{SingerCucuringu})
By Proposition~\ref{prop:cycle-rank2}, this space is one-dimensional, so the Jacobian has rank $2n-1$.
Consequently,
$\dim\cR(C_n)=2n-1=\dim A_{C_n}-1$.
Therefore, $\cR(C_n)$ is a hypersurface.
\end{proof}

\subsection{Recurrence formula for cycles}

In this subsection, let $C_m$ have cyclic order $v_1,\ldots,v_m$.
By Proposition~\ref{prop:cycle-essential-hypersurface},
for $m\ge4$, an irreducible defining polynomial $\rho_{C_m}\in R_{C_m}$ of $\cR(C_m)$ 
is  determined up to a nonzero scalar. 
We may also denote $\rho_{C_3}:=\det X[\{v_1,v_2,v_3\}]$.
For $m\ge4$, Proposition~\ref{prop:cycle-essential-hypersurface} gives
\begin{equation}
\cV_{A_{C_m}}(\rho_{C_m})
=\overline{\Phi_m(\Omega_m)}^{\mathrm{Zar}}.
\label{eq:cycle-gram-closure}
\end{equation}

We are ready to state the main theorem of this section.
\begin{thm}
\label{thm:cycle-resultant-recursion}
Let $n$ be an integer with $n\ge3$, and let $\rho_{C_n}$ be as defined above.
For every edge $\{v_i,v_{i+1}\}\in E(C_n)$,
\begin{equation}
\deg_{x_{i,i+1}}\rho_{C_n}=2^{n-2},
\qquad
\operatorname{lc}_{x_{i,i+1}}\rho_{C_n}
\doteq\prod_{\substack{1\le k\le n\\k\notin\{i,i+1\}}}x_{kk}^{\,2^{n-3}},
\label{eq:cycle-edge-degree-leading}
\end{equation}
where $\operatorname{lc}_{x_{ij}}$ denotes the leading coefficient as a polynomial in $x_{ij}$.
Moreover, if  $n\ge4$, then $\rho_{C_n}$ 
 can be scaled so that
\begin{equation*}
\rho_{C_n}
=\operatorname{Res}_{x_{1,n-1}}\!\left(
\det X[\{v_{n-1},v_n,v_1\}],\,\rho_{C_{n-1}}
\right).
\end{equation*}
\end{thm}

Before the proof of Theorem~\ref{thm:cycle-resultant-recursion},
we first establish the degree lower bound used in the proof.
\begin{lemma}
\label{lem:cycle-edge-degree-lower-bound}
Let $n$ be an integer with $n\ge4$. Then, for every edge $\{v_i,v_{i+1}\}\in E(C_n)$,
\[
\deg_{x_{i,i+1}}\rho_{C_n}\ge2^{n-2}.
\]
\end{lemma}

\begin{proof}
Label the vertices in cyclic order.
By cyclic symmetry, it suffices to consider $x_{1n}$.
Choose generic positive angles $\theta_1,\ldots,\theta_{n-1}$ such that
\begin{equation}\label{eq:theta}
\theta_1 >> \theta_2 >>\dots >> \theta_{n-1}>0,
\qquad
\sum_{j=1}^{n-1}\theta_j<\pi,
\end{equation}
where $\theta_i >> \theta_{i+1}$ means that 
$\theta_i$ is sufficiently larger than $\theta_{i+1}$.
Let $h$ be a polynomial in $x_{1n}$ obtained from $\rho_{C_n}$ by 
putting $x_{ii}=1$ for all $i$ and $x_{i,i+1}=\cos \theta_i$ for $i=1,\dots, n-1$.
Then, $\deg_{x_{1n}}\rho_{C_n}\ge2^{n-2}$ follows by showing 
that $h$ has at least $2^{n-2}$ distinct roots.
We first check that $h$ is nonzero.
Since $\cV(\rho_{C_n})$ is a proper hypersurface invariant under invertible diagonal scaling, while the Zariski-dense set of points in $A_{C_n}$ with nonzero diagonal entries can be normalized to $x_{ii}=1$, setting $x_{ii}=1$ for all $i$ does not make $\rho_{C_n}$ vanish identically.
%
Also, since the admissible tuples $(\cos \theta_1, \dots, \cos \theta_{n-1})$ form a nonempty open subset of 
$\mathbb{R}^{n-1}$,  
the generic choice of the angles gives $h\not\equiv0$.

To see that $h$ has at least $2^{n-2}$ distinct roots, pick any sign vector $\varepsilon=(\varepsilon_2,\ldots,\varepsilon_{n-1})\in\{\pm1\}^{n-2}$.
Define $U^{\varepsilon}=(u_1^{\varepsilon}, \dots, u_n^{\varepsilon})\in (\mathbb{C}^2)^n$ by 
\[
u_i^{\varepsilon}=(\cos \varphi_i^{\varepsilon}, \sin \varphi_i^{\varepsilon})^{\top}
\]
with
\[
\varphi_1^{\varepsilon}=0, \quad \varphi_2^{\varepsilon}=\theta_1, 
\quad
\varphi_i^{\varepsilon}=\theta_1+\sum_{j=2}^{i-1}\varepsilon_j\theta_j
\quad(3\le i\le n).
\]
Then $(u_i^{\varepsilon})^Tu_i^{\varepsilon}=1$ and $(u_i^{\varepsilon})^Tu_{i+1}^{\varepsilon}=\cos\theta_i$ for $1\le i\le n-1$, independently of $\varepsilon$.
Moreover, the values $(u_1^{\varepsilon})^T u_n^{\varepsilon}$ are distinct for all $\varepsilon$ by (\ref{eq:theta}).
Indeed, if two sign vectors first differ at position $j$, 
the difference contributed by $2\theta_j$ cannot be canceled by the later angles.
Hence their final angles $\varphi_n^{\varepsilon}$ are distinct,
and the injectivity of cosine in the interval $[0, \pi]$ implies 
that the values $(u_1^{\varepsilon})^T u_n^{\varepsilon}$ are all distinct.
Since 
$\Phi_n(U^{\varepsilon})\in\cV_{A_{C_n}}(\rho_{C_n})$,
$(u_1^{\varepsilon})^T u_n^{\varepsilon}$ is a root of $h$,
and $h$ has at least $2^{n-2}$ distinct roots.
\end{proof}

\begin{proof}[Proof of Theorem~\ref{thm:cycle-resultant-recursion}]
We proceed by induction on $n$.
For $n=3$, we have $\rho_{C_3}=\det X[\{v_1,v_2,v_3\}]$ by definition
and the claim can be directly checked by expanding the determinant.
Let $n\ge4$ and assume the assertion for $n-1$.
Deleting $v_n$ and adding the edge $\{v_1,v_{n-1}\}$ gives $C_{n-1}$. Let
\[
\xi:=x_{1,n-1},\quad
f:=\det X[\{v_1,v_{n-1},v_n\}],\quad
g:=\rho_{C_{n-1}},\quad F=\operatorname{Res}_{\xi}(f,g).
\]
%
\begin{claim}\label{claim:cycle1}
$\cV_{A_{C_n}}(\rho_{C_n})\subseteq\cV_{A_{C_n}}(F)$.
\end{claim}
\begin{proof}
In view of~\eqref{eq:cycle-gram-closure},
it suffices to show that 
$\Phi_n(\Omega_n)\subseteq \cV_{A_{C_n}}(F)$.
Moreover, $\{U\in\Omega_n : \det(u_1,u_{n-1})\neq 0\}$ 
is a nonempty Zariski open subset of $\Omega_n$,
and hence it suffices to show that 
$U\in\Omega_n$ with $\det(u_1,u_{n-1})\neq0$ satisfies  
$\Phi_n(U)\in \cV_{A_{C_n}}(F)$.
To see this, observe that 
$U_f:=(u_1,u_{n-1},u_n)\in \Omega_3$ and 
$U_g:=(u_1,u_2,\dots, u_{n-1})\in \Omega_{n-1}$,
and hence $(\Phi_n(U), u_1^T u_{n-1})$ is a common solution of $f$ and $g$
in $A_{C_n+\{v_1,v_{n-1}\}}$.
Hence $F(\Phi_n(U))=0$ as required.
\end{proof}
\begin{claim}\label{claim:cycle2}
For $i=1,\dots,n-2$,
$\deg_{x_{i,i+1}} F=2^{n-2}$ and 
$\operatorname{lc}_{x_{i,i+1}} F
\doteq\prod_{\substack{1\le k\le n\\k\notin\{i,i+1\}}}x_{kk}^{\,2^{n-3}}$.
\end{claim}
\begin{proof}
Note that $f$ has degree two in $\xi=x_{1,n-1}$, 
with $\operatorname{lc}_{\xi} f \doteq x_{nn}$,
and the induction hypothesis gives $\deg_\xi g=2^{n-3}$.
Let $\eta_1,\eta_2$ be the roots of $f$ 
in an algebraic closure of the fraction field of $R_{C_n}$.
The product formula for the resultant (see, e.g.,\cite[Chapter~3, Section~1, formula~(1.4)]{CLOUsing}) gives
\begin{equation}\label{eq:F}
F\doteq(x_{nn})^{2^{n-3}} g(\eta_1)g(\eta_2).
\end{equation}
Fix $1\le i\le n-2$.
The polynomial $f$, and hence its roots, is independent of $x_{i,i+1}$.
By the induction hypothesis,
\[
\operatorname{lc}_{x_{i,i+1}}g
\doteq\prod_{\substack{1\le k\le n-1\\k\notin\{i,i+1\}}}x_{kk}^{\,2^{n-4}}.
\]
In particular, this leading coefficient is independent of $\xi$,
so substituting $\eta_1$ or $\eta_2$ for $\xi$ preserves the degree and leading coefficient in $x_{i,i+1}$.
Thus, \eqref{eq:F} and the induction hypothesis give
\begin{align*}
&\deg_{x_{i,i+1}} F=2(\deg_{x_{i,i+1}} g)=2^{n-2}, \\*
&\operatorname{lc}_{x_{i,i+1}}F
\doteq(x_{nn})^{2^{n-3}} \left(\operatorname{lc}_{x_{i,i+1}}g\right)^2
\doteq\prod_{\substack{1\le k\le n\\k\notin\{i,i+1\}}}x_{kk}^{\,2^{n-3}}.\qedhere
\end{align*}
\end{proof}

By Claim~\ref{claim:cycle1} and the irreducibility of $\rho_{C_n}$, we may write
$F=H\rho_{C_n}$ for some nonzero $H\in R_{C_n}$.
Lemma~\ref{lem:cycle-edge-degree-lower-bound} and Claim~\ref{claim:cycle2} imply
\[
2^{n-2}\le\deg_{x_{i,i+1}}\rho_{C_n}
\le\deg_{x_{i,i+1}}F=2^{n-2}.
\]
This implies that $H$ is independent of each common edge variable,
and comparing leading coefficients in $F=H\rho_{C_n}$ yields
\[
H\ \bigm|\ \prod_{\substack{1\le k\le n\\k\notin\{i,i+1\}}}x_{kk}^{2^{n-3}}
\qquad(1\le i\le n-2).
\]
The greatest common divisor of these monomials is $x_{nn}^{2^{n-3}}$.
Consequently, $H\doteq x_{nn}^{r}$ for some $0\le r\le 2^{n-3}$.
The leading coefficient of $F$ in $x_{23}$ has degree $2^{n-3}$ in $x_{11}$.
Since $H$ is independent of $x_{11}$, 
$F=H\rho_{C_n}$  and Claim~\ref{claim:cycle2} imply that
$\operatorname{lc}_{x_{23}}\rho_{C_n}$
has degree $2^{n-3}$ in $x_{11}$.
By the symmetry of $\rho_{C_n}$, $\operatorname{lc}_{x_{12}}\rho_{C_n}$ has degree $2^{n-3}$ in $x_{nn}$.
Comparing the $x_{nn}$-degrees in
$\operatorname{lc}_{x_{12}}F$ and 
$\operatorname{lc}_{x_{12}} H\rho_{C_n}$,
we obtain $r=0$.
Therefore $F\doteq\rho_{C_n}$.

By the symmetry of $\rho_{C_n}$, Claim~\ref{claim:cycle2} now gives the asserted degree and leading coefficient, up to a nonzero scalar, for every edge of $C_n$.
We may thus normalize $\rho_{C_n}$ by setting $\rho_{C_n}=F$, which proves the recurrence and completes the proof.
\end{proof}

\subsection{Iterated resultants along a triangulation}

Recall the iterated resultant defined in Subsection~\ref{subsec:iterated-resultants}.
The recurrence of Theorem~\ref{thm:cycle-resultant-recursion} gives the following corollary, which proves Theorem~\ref{thm:cycle-iterated-resultant}.

\begin{cor}
\label{cor:iterated-resultant-essential}
Let $n$ be an integer with $n\ge4$, and let $T$ be a triangulation of $C_n$.
For every ear-removal order,
$
\operatorname{ItrRes}(C_n)\doteq\rho_{C_n}.
$
Thus the iterated resultant is independent of the triangulation and the ear-removal order up to a nonzero scalar.
In particular,
\[
\cD(C_n)=
\bigcup_{e\in E(C_n)}
\cV_{A_{C_n}}(\det X[e])
\ \cup\ \cV_{A_{C_n}}\bigl(\operatorname{ItrRes}(C_n)\bigr).
\]
\end{cor}

\begin{proof}
We proceed by induction on the number of vertices, with the triangle case following from the definition of $\rho_{C_3}$.
Let $\{v_s,v_i,v_t\}$ be the first ear in the chosen order, and let $C'$ and $T'$ be obtained by removing its ear vertex $v_i$.
By induction, $\operatorname{ItrRes}(C')\doteq\rho_{C'}$.
The recursive definition and Theorem~\ref{thm:cycle-resultant-recursion}, applied after relabeling, give
\[
\operatorname{ItrRes}(C_n)
\doteq\operatorname{Res}_{x_{st}}\!\left(
\det X[\{v_s,v_i,v_t\}],\,\rho_{C'}
\right)
\doteq\rho_{C_n}.
\]
Since $\rho_{C_n}$ is an irreducible defining polynomial of $\cR(C_n)$ for $n\ge4$, the identity above proves Theorem~\ref{thm:cycle-iterated-resultant}.
The formula for $\cD(C_n)$ follows from Proposition~\ref{prop:max-clique-component} and the definition of $\cR(C_n)$.
\end{proof}

\begin{remark}
The same iterated resultant formula 
and the degree formula  are obtained for $\partial_a\mathscr C(C_n)$ in \cite{LaurentMascarinTelen}.
Their proof is based on an interesting connection between
$\partial_a\mathscr C(C_n)$ and Lissajous varieties,
whereas our proof derives these formulas directly from the stress
characterization of cycle degeneracy.
\end{remark}

\section{MC discriminants of $K_4$-minor-free graphs \mbox{and cycle-completable graphs}}
\label{sec:K4-atoms}
In this section, we prove Theorem~\ref{thm:main}.
Subsection~\ref{subsec:cycle-lifting} shows that the essential discriminant of each chordless cycle of a $K_4$-minor-free graph lifts to the MC discriminant of the whole graph.
Subsection~\ref{subsec:sp-reconstruction} introduces SP-trees and develops reconstruction arguments for stresses and completions.
Then Subsection~\ref{sec:K4} combines these results with the clique-sum formula to prove the main theorem.

\subsection{Lifting the cycle essential discriminant to the whole graph}
\label{subsec:cycle-lifting}

We begin with the following theorem, which shows that, for a $K_4$-minor-free graph, the inverse image of the essential discriminant of each chordless cycle of length at least four is contained in the MC discriminant of the whole graph.

\begin{thm}
\label{thm:L}
Let $G$ be a $K_4$-minor-free graph. For every chordless cycle $C\subset G$,
\[
 p_{G,C}^{-1}\cR(C)\subset \cD(G).
\]
\end{thm}

\begin{proof} 
If $C=K_3$, then $\cR(C)=\varnothing$, so there is nothing to prove. Assume $|V(C)|\ge4$.
For each $i\in V(G)\setminus V(C)$, let $S_i$ be the set of vertices in $V(C)$ adjacent to $i$ in $G$.

\begin{claim}
\label{claim:cycle-neighbor-bound}
For every $i\in V(G)\setminus V(C)$, we have $|S_i|\le2$.
\end{claim}
\begin{proof}
Suppose that some $i\in V(G)\setminus V(C)$ is adjacent to three distinct vertices $s,t,r$ of $C$.
The contraction of the three paths of $C$ with endvertices $s,t,r$ to edges results in a triangle on $s,t,r$.
Since $i$ is adjacent to all three vertices, this yields a $K_4$-minor, a contradiction.
\end{proof}

Let
\[
\Omega':=\left\{
U=(u_i)_{i\in V(C)}\in(\C^2)^{V(C)}:
\det(u_i,u_j)\neq0\ \text{for all distinct }i,j
\right\},
\]
and define $\Phi_C(U):=p_C(U^TU)$ for $U\in\Omega'$.
After relabeling the vertices of $C$, the set $\Omega'$ is a nonempty Zariski open subset of $\Omega_{|V(C)|}$, and hence is dense in it.
By Proposition~\ref{prop:cycle-essential-hypersurface} and the density of $\Omega'$ in $\Omega_{|V(C)|}$, it suffices to show
$
p_{G,C}^{-1}\bigl(\Phi_C(\Omega')\bigr)\subseteq\Sigma(G)$.

Let $a\in p_{G,C}^{-1}\bigl(\Phi_C(\Omega')\bigr)$.
Then there exists
$U=(u_j)_{j\in V(C)}\in\Omega'$ such that
$p_{G,C}(a)=p_C(U^TU)$. Let $Y_C:=U^TU$. 
Since $\rank Y_C = 2$, Proposition~\ref{prop:cycle-rank2}
implies that there exists $Z_C\in \S_{C}\setminus\{O\}$ such that
$Y_CZ_C=O$. Let $\widetilde Z_C$ be the extension of $Z_C$ by zero to a matrix in $\S_{G}$.
Now we construct a completion $X$ of $a$ satisfying $X\widetilde Z_C=O$. First, set
$X[V(C)]:=Y_C$. Let $i\in V(G)\setminus V(C)$. By Claim~\ref{claim:cycle-neighbor-bound}, we have $|S_i|\le2$.

If $|S_i|=2$, write $S_i=\{s,t\}$. Since $U\in\Omega'$, we have $\det Y_C[S_i]=\det(u_s,u_t)^2\neq0$, so $Y_C[S_i]$ is invertible. Using row vector
$
(a_{is},\ a_{it})
$, 
define
\[
X[i,V(C)]:=
\begin{pmatrix}
a_{is} & a_{it}
\end{pmatrix}
\,Y_C[S_i]^{-1}\,Y_C[S_i,V(C)].
\]
Then $X[i,S_i]=
\begin{pmatrix}
a_{is} & a_{it}
\end{pmatrix}
\,Y_C[S_i]^{-1}\,Y_C[S_i,S_i]=
(a_{is},\ a_{it})
$, so $X$ satisfies the prescribed entries by $a$ on $is$ and $it$. 
Moreover, $X[i,V(C)]$ is a linear combination of rows of $Y_C$, and hence
$X[i,V(C)]Z_C=O$.

If $|S_i|=1$, write $S_i=\{s\}$ and choose one neighbor $q$ of $s$ on $C$. Since $\{s,q\}\in E(C)$, 
we see that $Y_C[\{s,q\}]$ is invertible. Define
\[
X[i,V(C)]:=
\begin{pmatrix}
 a_{is} & 0
\end{pmatrix}
\,Y_C[\{s,q\}]^{-1}\,Y_C[\{s,q\},V(C)].
\]
Then $X[i,\{s,q\}]=(a_{is},0)$, so $X$ satisfies the prescribed entry by $a$ on $is$.
Again $X[i,V(C)]Z_C=O$.
If $|S_i|=0$, set $X[i,V(C)]:=0$. 

For the remaining entries of
$X[V(G)\setminus V(C),V(G)\setminus V(C)]$, insert the diagonal and edge entries prescribed by $a$ and choose the nonedge entries arbitrarily. 
Then $X[V(G),V(C)]Z_C=O$, and consequently $X\widetilde Z_C=O$. 
Therefore, $a = p_G(X)\in\Sigma(G)$.
\end{proof}


\subsection{Reconstructing stresses and completions using SP-trees}
\label{subsec:sp-reconstruction}

This subsection develops two reconstruction arguments (Lemma~\ref{lem:2} and Lemma~\ref{prop:9}) that form a central part of the proof of our main theorem.
The combinatorial tool underlying these arguments is the SP-tree.
We first introduce SP-trees and their relevant properties, and then prove the reconstruction results.

\subsubsection{SP-trees}
\label{subsec:sp-tree}

SP-trees describe the decomposition of biconnected series--parallel graphs (equivalently, biconnected $K_4$-minor-free graphs) into cycles and multigraphs on two vertices.
They are the S/P specialization of SPQR-trees; see~\cite{ChimaniHlineny,KriegeKurpiczMutzel}.
We use an equivalent formulation in which all skeletons are subgraphs of a common multigraph $H$.

\begin{defn}
\label{def:sp-tree}
Let $G=(V,E)$ be a biconnected $K_4$-minor-free graph with $|V|\ge3$.
An \emph{SP-tree} of $G$ consists of a finite tree $\mathcal T$, a loopless multigraph $H$ containing $G$, with
\[
V(H)=V,\qquad
E(H)=E\sqcup\{e_{\mu\nu}:\mu\nu\in E(\mathcal T)\},
\]
and a subgraph $\operatorname{skel}(\mu)$ of $H$, called the \emph{skeleton} of $\mu$, for each node $\mu\in V(\mathcal T)$.
The edges in $E$ are called \emph{real}, and the added edges $e_{\mu\nu}$ are called \emph{virtual}.
These data satisfy the following conditions.
\begin{enumerate}
\item Each node is either an \emph{S-node}, whose skeleton is a simple cycle of length at least three, or a \emph{P-node}, whose skeleton consists of two vertices joined by at least three parallel edges. Adjacent nodes have different types.
\item Each real edge belongs to exactly one skeleton, and each virtual edge $e_{\mu\nu}$ belongs to exactly the two skeletons $\operatorname{skel}(\mu)$ and $\operatorname{skel}(\nu)$.
\item For each tree edge $\mu\nu$, let $\mathcal T_\mu$ and $\mathcal T_\nu$ be the two components of $\mathcal T-\mu\nu$, containing $\mu$ and $\nu$, respectively. If $p,q$ are the endpoints of $e_{\mu\nu}$, then
\[
\left(\bigcup_{\lambda\in V(\mathcal T_\mu)}V(\operatorname{skel}(\lambda))\right)
\cap
\left(\bigcup_{\lambda\in V(\mathcal T_\nu)}V(\operatorname{skel}(\lambda))\right)
=\{p,q\}.
\]
\end{enumerate}
The two vertices of a P-skeleton are called its \emph{poles}.
\end{defn}

Every biconnected $K_4$-minor-free graph with at least three vertices admits an SP-tree; see~\cite[Section~4.1]{KriegeKurpiczMutzel}.
Figure~\ref{fig:spqr-example} shows $G$, the common multigraph $H$, and the tree $\mathcal T$ with its skeletons.


\definecolor{skelA}{RGB}{230,159,0}   
\definecolor{skelB}{RGB}{0,158,115}   
\definecolor{skelC}{RGB}{0,114,178}   
\definecolor{skelD}{RGB}{213,94,0}    
\definecolor{skelE}{RGB}{204,121,167} 

\begin{figure}[!htbp]
\centering
\resizebox{\textwidth}{!}{%
\begin{tikzpicture}[
  vertex/.style={
    circle,
    draw,
    fill=white,
    inner sep=1.2pt,
    minimum size=5mm
  },
  component/.style={
    draw=orange!80!black,
    rounded corners=3pt,
    line width=1.0pt
  },
  treeedge/.style={
    draw=black,
    line width=1.0pt
  },
  virtual/.style={
    draw=black,
    dash pattern=on 0pt off 3pt,
    line cap=round,
    line width=1.6pt
  },
  A/.style={draw=skelA,line width=2.0pt},
  B/.style={draw=skelB,line width=2.0pt},
  C/.style={draw=skelC,line width=2.0pt},
  D/.style={draw=skelD,line width=2.0pt},
  E/.style={draw=skelE,line width=2.0pt},
  baglabel/.style={
    font=\small,
    inner sep=0pt
  }
]


\begin{scope}[shift={(0,0)}]

  \node at (0,-6.9) {$G$};

  \node[vertex] (v1) at (0,-0.8) {$1$};
  \node[vertex] (v2) at (0,-6.2) {$2$};

  \node[vertex] (v3) at (-2.0,-3.5) {$3$};

  \node[vertex] (v4) at (-0.55,-2.15) {$4$};
  \node[vertex] (v5) at (-0.55,-4.45) {$5$};
  \node[vertex] (v6) at ( 0.70,-2.75) {$6$};
  \node[vertex] (v7) at ( 0.70,-3.85) {$7$};

  \node[vertex] (v8) at (2.0,-2.6) {$8$};
  \node[vertex] (v9) at (2.0,-4.4) {$9$};

  \draw[A] (v1) -- (v3) -- (v2);

  \draw[B] (v1) -- (v4);
  \draw[B] (v5) -- (v2);

  \draw[D] (v4) -- (v6) -- (v5);

  \draw[E] (v4) -- (v7) -- (v5);

  \draw[C] (v1) -- (v8) -- (v9) -- (v2);

\end{scope}


\begin{scope}[shift={(7.8,0)}]

  \node at (-0.8,-6.9) {$H$};

  \node[vertex] (h1) at (0,-0.8) {$1$};
  \node[vertex] (h2) at (0,-6.2) {$2$};
  \node[vertex] (h3) at (-2.0,-3.5) {$3$};
  \node[vertex] (h4) at (-0.55,-2.15) {$4$};
  \node[vertex] (h5) at (-0.55,-4.45) {$5$};
  \node[vertex] (h6) at (0.70,-2.75) {$6$};
  \node[vertex] (h7) at (0.70,-3.85) {$7$};
  \node[vertex] (h8) at (2.0,-2.6) {$8$};
  \node[vertex] (h9) at (2.0,-4.4) {$9$};

  \draw[A] (h1) -- (h3) -- (h2);
  \draw[B] (h1) -- (h4);
  \draw[B] (h5) -- (h2);
  \draw[D] (h4) -- (h6) -- (h5);
  \draw[E] (h4) -- (h7) -- (h5);
  \draw[C] (h1) -- (h8) -- (h9) -- (h2);

  \draw[virtual] (h1) .. controls (-3.5,-0.8) and (-3.5,-6.2) .. (h2);
  \draw[virtual] (h1) .. controls (-4.1,-0.8) and (-4.1,-6.2) .. (h2);
  \draw[virtual] (h1) .. controls (-4.7,-0.8) and (-4.7,-6.2) .. (h2);
  \draw[virtual] (h4) -- (h5);
  \draw[virtual] (h4) .. controls (-0.90,-2.8) and (-0.90,-3.8) .. (h5);
  \draw[virtual] (h4) .. controls (-1.25,-2.8) and (-1.25,-3.8) .. (h5);

\end{scope}


\begin{scope}[shift={(11.5,0)}]

  \node at (4.7,-8.9) {$\mathcal T$};

  \draw[treeedge] (4.7,-1.8) -- (1.2,-2.6);
  \draw[treeedge] (4.7,-1.8) -- (4.7,-2.6);
  \draw[treeedge] (4.7,-1.8) -- (8.2,-2.6);

  \draw[treeedge] (4.7,-4.4) -- (4.7,-4.8);

  \draw[treeedge] (4.7,-6.0) -- (2.6,-6.6);
  \draw[treeedge] (4.7,-6.0) -- (6.8,-6.6);

  \draw[component]
    (3.0,-1.8) rectangle (6.4,-0.7);

  \node[baglabel,anchor=south west]
    at (3.1,-0.6) {$P$};

  \node[vertex] (p1) at (3.8,-1.25) {$1$};
  \node[vertex] (p2) at (5.6,-1.25) {$2$};

  \draw[virtual] (p1) to[bend left=22]  (p2);
  \draw[virtual] (p1) -- (p2);
  \draw[virtual] (p1) to[bend right=22] (p2);

  \draw[component]
    (0.0,-4.2) rectangle (2.4,-2.6);

  \node[baglabel,anchor=south west]
    at (0.1,-2.5) {$S$};

  \node[vertex] (a1) at (0.6,-3.05) {$1$};
  \node[vertex] (a2) at (1.8,-3.05) {$2$};
  \node[vertex] (a3) at (1.2,-3.75) {$3$};

  \draw[virtual] (a1) -- (a2);
  \draw[A] (a1) -- (a3) -- (a2);

  \draw[component]
    (3.1,-4.35) rectangle (6.3,-2.6);

  \node[baglabel,anchor=south west]
    at (3.2,-2.5) {$S$};

  \node[vertex] (b1) at (3.7,-2.95) {$1$};
  \node[vertex] (b2) at (5.7,-2.95) {$2$};
  \node[vertex] (b4) at (3.8,-3.90) {$4$};
  \node[vertex] (b5) at (5.6,-3.90) {$5$};

  \draw[virtual] (b1) -- (b2);
  \draw[virtual] (b4) -- (b5);

  \draw[B] (b1) -- (b4);
  \draw[B] (b5) -- (b2);

  \draw[component]
    (7.0,-4.2) rectangle (9.4,-2.6);

  \node[baglabel,anchor=south west]
    at (9.0,-2.5) {$S$};

  \node[vertex] (c1) at (7.6,-3.05) {$1$};
  \node[vertex] (c2) at (8.8,-3.05) {$2$};
  \node[vertex] (c8) at (7.6,-3.75) {$8$};
  \node[vertex] (c9) at (8.8,-3.75) {$9$};

  \draw[virtual] (c1) -- (c2);
  \draw[C] (c1) -- (c8) -- (c9) -- (c2);

  \draw[component]
    (3.3,-6.0) rectangle (6.1,-4.8);

  \node[baglabel,anchor=south west]
    at (3.4,-4.7) {$P$};

  \node[vertex] (q4) at (4.0,-5.35) {$4$};
  \node[vertex] (q5) at (5.4,-5.35) {$5$};

  \draw[virtual] (q4) to[bend left=22]  (q5);
  \draw[virtual] (q4) -- (q5);
  \draw[virtual] (q4) to[bend right=22] (q5);

  \draw[component]
    (1.4,-8.2) rectangle (3.8,-6.6);

  \node[baglabel,anchor=south west]
    at (1.5,-6.5) {$S$};

  \node[vertex] (d4) at (2.0,-7.05) {$4$};
  \node[vertex] (d5) at (3.2,-7.05) {$5$};
  \node[vertex] (d6) at (2.6,-7.75) {$6$};

  \draw[virtual] (d4) -- (d5);
  \draw[D] (d4) -- (d6) -- (d5);

  \draw[component]
    (5.6,-8.2) rectangle (8.0,-6.6);

  \node[baglabel,anchor=south west]
    at (7.6,-6.5) {$S$};

  \node[vertex] (e4) at (6.2,-7.05) {$4$};
  \node[vertex] (e5) at (7.4,-7.05) {$5$};
  \node[vertex] (e7) at (6.8,-7.75) {$7$};

  \draw[virtual] (e4) -- (e5);
  \draw[E] (e4) -- (e7) -- (e5);

\end{scope}

\end{tikzpicture}%
}
\caption{A $K_4$-minor-free graph $G$ (left), the common multigraph $H$ (middle), and the SP-tree $\mathcal T$ (right). Real edges belonging to the same skeleton are shown in the same color.
Virtual edges are dotted.}
\label{fig:spqr-example}
\end{figure}
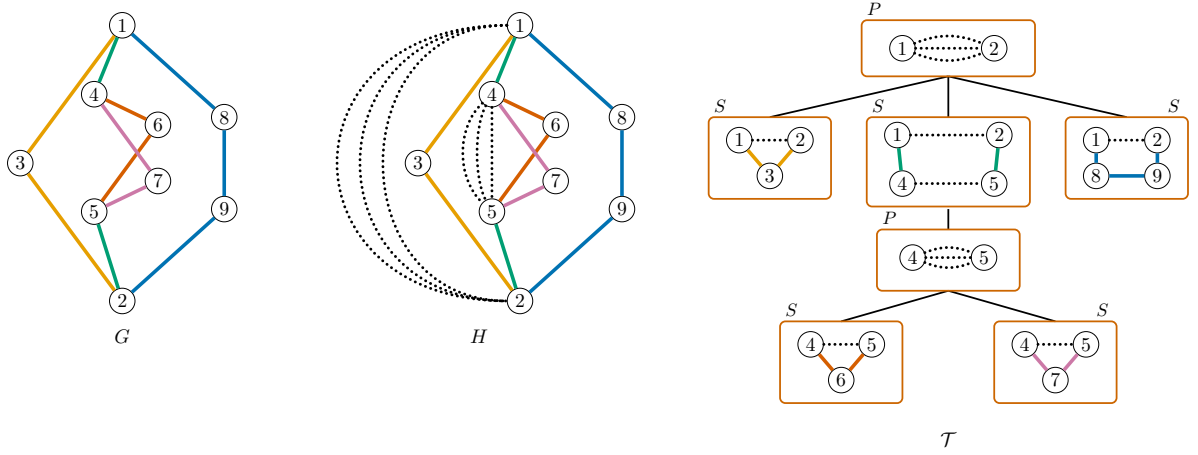

We record the properties needed below.

\begin{lemma}
\label{lem:sp-tree-properties}
Let $G$ be a biconnected $K_4$-minor-free graph with at least three vertices, and let $\cT$ be its SP-tree. Then the following hold.
\begin{enumerate}[(i)]
\item \label{sp:branches}
Let $\tau$ be a P-node with poles $p,q$, and let $\cT_1,\ldots,\cT_m$ be the connected components of $\cT-\tau$.
For each $i$, set
\[
H_i:=G\cap\left(\bigcup_{\mu\in V(\cT_i)}\operatorname{skel}(\mu)\right).
\]
Then 
$
V(H_i)\setminus\{p,q\}\ne\varnothing$, and
$
V(H_i)\cap V(H_j)=\{p,q\} \  (i\ne j)$.
Their union $\bigcup_{i=1}^m H_i$ is $G-pq$ if $pq\in E(G)$, and $G$ otherwise. In particular, $\{p,q\}$ is a $2$-separator of $G$.
\item \label{sp:p-degree}
Every P-node has degree at least two in $\cT$, and at least three if $G$ has no clique separator of size two.
\item \label{sp:leaf-path}
If $G$ is not a cycle, every leaf of $\cT$ is an S-node. For each such leaf $\mu$, the skeleton $\operatorname{skel}(\mu)$ consists of a unique virtual edge $pq$ and a real $p$--$q$ path of length at least two. The internal vertices of this path have degree two in $G$.
\item \label{sp:two-leaves}
If $G$ is not a cycle and has no clique separator of size two, some P-node is adjacent to at least two leaf S-nodes.
\item \label{sp:triangle-free}
If $G$ is a clique-sum atom other than $K_3$, then $G$ is triangle-free.
\end{enumerate}
\end{lemma}

\begin{proof}
For~(i) and~(ii), 
let $\tau$ be a P-node with poles $p,q$. 
By the condition (1), 
$\operatorname{skel}(\tau)$ consists of at least three parallel edges, 
at most one of which is real since $G$ is simple. 
Each virtual edge in $\operatorname{skel}(\tau)$ corresponds to a tree edge incident to $\tau$ in $\cT$,  
so $\tau$ has degree at least two in $\cT$ and has degree at least three if $G$ has no edge between $p$ and $q$. 
Since every real edge belongs to exactly one skeleton by condition (2), 
the union of the $H_i$ is either $G-pq$ or $G$.
Also, the condition (3) gives $V(H_i)\cap V(H_j)=\{p,q\}$ for $i\ne j$. 
Each $H_i$ also contains a vertex other than $p,q$ of the S-skeleton adjacent to $\tau$ in $\cT_i$
since this skeleton is a simple cycle.
Thus, $\{p,q\}$ is a $2$-separator of $G$, and $G$ has no edge $pq$ 
if $G$ has no clique separator of size two. 

It remains to show that, in the case where $pq\in E(G)$, the real edge $pq$ belongs to 
$\operatorname{skel}(\tau)$ (and hence the union of the $H_i$ is $G-pq$).
Suppose the real edge $pq$ belongs to a skeleton $\operatorname{skel}(\mu)$ other than $\operatorname{skel}(\tau)$.
Then, consider the path in $\cT$ from $\tau$ to $\mu$.
Applying the condition (3) to each tree edge of this path,
we see that the virtual edge corresponding to each tree edge has endpoints $p$ and $q$.
In particular,  the S-skeleton next to $\tau$ in this path contains 
two parallel edges between $p$ and $q$.
This contradicts the fact that each S-skeleton is a simple cycle.

For~(iii), every leaf is an S-node by~(ii). 
Deleting its unique virtual edge leaves a real path of length at least two. 
By condition (3), this path meets the rest of $G$ only at its endpoints, 
so its internal vertices have degree two.

For~(iv), consider deleting all leaf S-nodes to obtain a nonempty tree $\cT'$. 
Every remaining S-node retains its degree, which is at least two. 
Thus a leaf of $\cT'$, or its only node if $\cT'$ is a singleton, is a P-node. 
Its degree in $\cT$ was at least three by~(ii), so at least two of its neighbors were deleted,
meaning that the P-node is adjacent to at least two leaf S-nodes.

For~(v), suppose that $G$ contains a triangle $\Delta$. 
Since $G\ne K_3$, choose a component $C$ of $G-V(\Delta)$. 
It has exactly two neighbors in $\Delta$ by biconnectivity and the $K_4$-minor-free property. 
Then these two neighbors  form a clique separator separating $C$ from the third vertex of $\Delta$.
\end{proof}

\subsubsection{Reconstructing stresses and completions}

For an $X$-stress $Z$, its \emph{vertex support} is defined by
$V(\supp Z):=\{i\in V:Z_{ij}\ne0\text{ for some }j\in V\}$.
An $X$-stress is said to have \emph{full support} if $V(\supp Z)=V$.
A nonzero $X$-stress $Z$ is said to be \emph{vertex-minimal} if no nonzero $X$-stress has vertex support properly contained in $V(\supp Z)$.
%
%

For $X\in\S^V$, we say that a pair $\{p,q\}\subseteq V$ \emph{collapses in $X$} if the $p$-th and $q$-th columns of $X$ are linearly dependent. If $X=U^TU$ is a minimal-rank Gram representation, this is equivalent to the linear dependence of $u_p,u_q$, since $U^T$ is injective. In particular, this condition is independent of the choice of a minimal-rank Gram representation.

For a biconnected $K_4$-minor-free graph $G$ with at least three vertices and no clique separator of size two, the subgraphs $H_i$ in Lemma~\ref{lem:sp-tree-properties} satisfy the assumptions of the following lemma: their poles are nonadjacent by~Lemma~\ref{lem:sp-tree-properties}(\ref{sp:branches}), and $m\ge3$ by~Lemma~\ref{lem:sp-tree-properties}(\ref{sp:p-degree}). We state the lemma for an arbitrary graph with these properties.

\begin{lemma}
\label{lem:2}
Suppose that a graph $G=(V,E)$ is the union of subgraphs $H_1,\ldots,H_m$, $m\ge3$, having the same two vertices $p,q$ such that
\[
V(H_\ell)\cap V(H_{\ell'})=\{p,q\}\quad(\ell\neq\ell'),
\quad
V(H_\ell)\setminus\{p,q\}\neq\varnothing,
\quad
\{p,q\}\notin E(G).
\]
Let $X\in\S^V$, and let $Z\in\S_G\setminus\{O\}$ be a full-support $X$-stress. If $\{p,q\}$ does not collapse in $X$, then there exists a nonzero $X$-stress in $\S_G$ with strictly smaller vertex support.
\end{lemma}

\begin{proof}
Choose a minimal-rank Gram representation $X=U^TU$. By Lemma~\ref{lem:1}, we have $UZ=O$. Since $\{p,q\}$ does not collapse in $X$, the vectors $u_p,u_q$ are linearly independent.
Since $m\ge3$ and $Z$ has full vertex support, it suffices to construct a nonzero $U$-stress supported on
$V(H_1)\cup V(H_2)$.
Set $S:=\{p,q\},\ I_\ell:=V(H_\ell)\setminus S\ (\ell=1,2)$,
and $R:=V(G)\setminus(I_1\cup S\cup I_2)$.
We use the vertex order $(I_1,S,I_2,R)$.
Let $D_1,D_2$ be diagonal $2\times2$ matrices to be chosen below, and define\par
{\small %
$$
Z_1:=
\begin{pmatrix}
Z[I_1] & Z[I_1,S] & O & O\\
Z[S,I_1] & D_1 & O & O\\
O&O&O&O\\
O&O&O&O
\end{pmatrix},
\ 
Z_2:=
\begin{pmatrix}
O&O&O&O\\
O&D_2&Z[S,I_2]&O\\
O&Z[I_2,S]&Z[I_2]&O\\
O&O&O&O
\end{pmatrix}.
$$}%
Since $pq\notin E(G)$ and $D_1,D_2$ are diagonal, both $Z_1$ and $Z_2$
belong to $\S_{G}$.
The equilibrium equations for $Z$ at the vertices of $I_\ell$ give
\begin{equation}
U_{I_\ell}Z[I_\ell]+U_SZ[S,I_\ell]=O,
\qquad \ell=1,2,
\label{eq:branch-equilibrium}
\end{equation}
where $U_S=(u_p,u_q)$.
Therefore,
\begin{align*}
UZ_1 & = \bigl(O,\,U_{I_1}Z[I_1,S]+U_SD_1,\,O,\,O\bigr),\\
UZ_2 & = \bigl(O,\,U_{I_2}Z[I_2,S]+U_SD_2,\,O,\,O\bigr).
\end{align*}

Now we show that there exist diagonal matrices $D_1,D_2$ and $(\lambda_1,\lambda_2)\neq (0,0)$ such that
$
 U\left(
\lambda_1Z_1+\lambda_2Z_2
\right) = O$.
For $\ell=1,2$, write
$U_{I_\ell}Z[I_\ell,S]=(x_\ell,y_\ell)$.
Multiplying \eqref{eq:branch-equilibrium} on the right by $U_{I_\ell}^T$
and using the symmetry of $Z[I_\ell]$, we obtain
\begin{equation}
u_px_\ell^T+u_qy_\ell^T = x_\ell u_p^T+y_\ell u_q^T.
\label{eq:equilibrium-symmetry}
\end{equation}
Since $u_p,u_q$ are linearly independent, there exists a left inverse
$L$ of $U_S$ with $LU_S=I_2$.
Multiplying \eqref{eq:equilibrium-symmetry} on the right by $L^T$ shows
that $x_\ell,y_\ell\in\Span\{u_p,u_q\}$. Hence
$(x_\ell,y_\ell)=U_SA_\ell$
for some $2\times2$ matrix $A_\ell$
and thus 
\[
U_{I_\ell}Z[I_\ell,S]+U_SD_\ell=U_S(A_\ell + D_\ell). 
\]
Here, substituting $(x_\ell,y_\ell)=U_SA_\ell$ into \eqref{eq:equilibrium-symmetry} gives
$U_SA_\ell^TU_S^T=U_SA_\ell U_S^T$.
Multiplying on the left by $L$ and on the right by $L^T$ yields
$A_\ell^T=A_\ell$. Thus we may write
$A_\ell=
\left(
\begin{smallmatrix}
 a_\ell&c_\ell\\
 c_\ell&d_\ell
\end{smallmatrix}
\right)$.
Now set
$ D_\ell:=
\left(
\begin{smallmatrix}
 -a_\ell&0\\
 0&-d_\ell
\end{smallmatrix}
\right)\ (\ell=1,2)$.
Then
$$
U_S(A_\ell+D_\ell)
=
c_\ell(u_q,u_p).
$$
Consequently,
$$
UZ_1=
\bigl(O,\ c_1(u_q,u_p),\ O,\ O\bigr),
\quad
UZ_2=
\bigl(O,\ c_2(u_q,u_p),\ O,\ O\bigr).
$$
Choose $(\lambda_1,\lambda_2)\neq(0,0)$ satisfying
$\lambda_1c_1+\lambda_2c_2=0$,
and let
$Z':=\lambda_1Z_1+\lambda_2Z_2$.
Then $UZ'=O$ and $Z'\in\S_G$, so $XZ'=O$ and $Z'$ is an $X$-stress.
It remains to check that $Z'\neq O$. Choose $\ell\in\{1,2\}$ with
$\lambda_\ell\neq0$. Since $Z$ has full vertex support and $I_\ell$ is
nonempty, at least one entry of
$Z[I_\ell,I_\ell\cup S]$
is nonzero. This block occurs only in $Z_\ell$, and hence cannot be
canceled by the other matrix. Therefore $Z'\neq O$.
Since $Z'$ is supported on $I_1\cup S\cup I_2$,
we have
$V(\operatorname{supp}Z')\subsetneq V(G)$,
as required.
\end{proof}

The following corollary applies this lemma to SP-trees.

\begin{cor}
\label{cor:5}
Let $G=(V,E)$ be a biconnected $K_4$-minor-free clique-sum atom with at least three vertices, and let $\cT$ be an SP-tree of $G$. Let $X\in\S^V$. If there exists a vertex-minimal full-support $X$-stress $Z\in\S_G$, then $\{p,q\}$ collapses in $X$ for the poles $p,q$ of every P-node of $\cT$.
\end{cor}

\begin{proof}
Let $\tau$ be a P-node with poles $p,q$.
By Lemma~\ref{lem:sp-tree-properties}(\ref{sp:branches}) and~(\ref{sp:p-degree}), the connected components of $\cT-\tau$ give a decomposition satisfying the hypotheses of Lemma~\ref{lem:2}.
If $\{p,q\}$ did not collapse in $X$, that lemma would give a nonzero $X$-stress with strictly smaller vertex support, contradicting the vertex-minimality of $Z$.
\end{proof}

Using the collapse condition obtained in Corollary~\ref{cor:5}
and the properties of SP-trees in
Subsubsection~\ref{subsec:sp-tree}, we prove the following lemma,
which is the main result of this subsection.

\begin{lemma}
\label{prop:9}
Let $G$ be a $K_4$-minor-free clique-sum atom, and suppose that $G$ is neither $K_1$, $K_2$, $K_3$, nor a cycle. Let $a\in A_G\setminus\cKD(G)$, $X$ be a completion of $a$, and $Z$ be a vertex-minimal full-support $X$-stress. Then there exists a chordless cycle $C\subset G$ such that $p_{G,C}(a)$ has a completion of rank at most $2$.
\end{lemma}

\begin{proof}
Let $\cT$ be an SP-tree of $G$. By Lemma~\ref{lem:sp-tree-properties}(\ref{sp:two-leaves}), there is a P-node $\tau$ adjacent to two leaf S-nodes $\mu_1,\mu_2$. Let $p,q$ be the poles of $\tau$.
For $\ell=1,2$, deleting the unique virtual edge $pq$ from $\operatorname{skel}(\mu_\ell)$ gives a path
\begin{equation}\label{eq:leaf-path}
P_\ell:\ p=r_0^{(\ell)},r_1^{(\ell)},\ldots,r_{m_\ell}^{(\ell)}=q
\end{equation}
of length $m_\ell\ge2$ in $G$, by Lemma~\ref{lem:sp-tree-properties}(\ref{sp:leaf-path}).

\begin{claim}
\label{claim:leaf-cycle-induced}
$C:=P_1\cup P_2$ is a chordless cycle of $G$.
\end{claim}
\begin{proof}
Since $\mu_1,\mu_2$ are leaves in $\cT$, each forms a connected component of $\cT-\tau$. Thus $P_1,P_2$ are the corresponding subgraphs $H_i$ in Lemma~\ref{lem:sp-tree-properties}(\ref{sp:branches}), so they meet only at $p,q$. Their union is therefore a simple cycle.
By Lemma~\ref{lem:sp-tree-properties}(\ref{sp:leaf-path}), every internal vertex of either path has degree two in $G$. 
Moreover, $\{p,q\}$ is a $2$-separator by Lemma~\ref{lem:sp-tree-properties}(\ref{sp:branches}), 
and hence $pq\notin E(G)$ because $G$ is a clique-sum atom. 
Consequently, $C$ has no chord.
\end{proof}

We show that $C$ has the required property. If $\rank X[V(C)]\le2$, then $X[V(C)]$ itself is the desired completion. Hence assume that $\rank X[V(C)]>2$. We will construct a new completion of rank two that preserves the diagonal entries and the entries on the edges of $C$.
Choose a minimal-rank Gram representation $X=U^TU$, with $U=(u_i)_{i\in V(G)}$. By Lemma~\ref{lem:1}, we have $UZ=O$.
By Lemma~\ref{lem:sp-tree-properties}(\ref{sp:triangle-free}), $G$ is triangle-free, so every edge is a maximal clique. Thus $a\notin\cKD(G)$ implies $\det X[e]\ne0$ for every edge $e$ of $G$, and consequently
\begin{equation}
\dim\operatorname{Span}\{u_i,u_j\}=2
\quad\text{for every }\{i,j\}\in E(G).
\label{eq:leaf-cycle-edge-independence}
\end{equation}
Recall the notation $P_\ell$ from~\eqref{eq:leaf-path}.
For $\ell=1,2$, let
$L_\ell:=\operatorname{Span}\{u_p,u_{r_1^{(\ell)}}\}$.
$L_{\ell}$ is two-dimensional by~\eqref{eq:leaf-cycle-edge-independence}.

\begin{claim}
\label{claim:leaf-path-span}
For $\ell=1,2$, we have $u_{r_k^{(\ell)}}\in L_\ell$ for every $0\le k\le m_\ell$.
\end{claim}
\begin{proof}
Fix $\ell\in\{1,2\}$ and write $r_k=r_k^{(\ell)}$.
At each internal vertex $r_k$, where $1\le k\le m_\ell-1$, the only incident edges in $G$ are the two path edges. Hence the equilibrium equation is
\[
Z_{r_kr_{k-1}}u_{r_{k-1}}+Z_{r_kr_k}u_{r_k}+Z_{r_kr_{k+1}}u_{r_{k+1}}=0.
\]
Since $Z$ has full support, these coefficients are not all zero. Thus the three vectors are linearly dependent. Consecutive vectors are linearly independent by~\eqref{eq:leaf-cycle-edge-independence}, so
$
\operatorname{Span}\{u_{r_{k-1}},u_{r_k}\}
=\operatorname{Span}\{u_{r_k},u_{r_{k+1}}\}
$.
It follows successively that all vectors on $P_\ell$ lie in $L_\ell$.
\end{proof}

\begin{claim}
\label{claim:leaf-path-isometry}
There exists a linear map $\Psi:L_2\to L_1$ such that
\[
\Psi(x)^T\Psi(y)=x^Ty\quad(x,y\in L_2),
\qquad \Psi(u_p)=u_p.
\]
\end{claim}
\begin{proof}
Recall that a bilinear form $f$ on a vector space $W$ is \emph{nondegenerate} if $f(x,y)=0$ for every $y\in W$ implies $x=0$.
For $\ell=1,2$, the matrix representing the bilinear form $x^Ty$ 
on $L_\ell$ in the basis $u_p,u_{r_1^{(\ell)}}$ is 
\[
X[\{p,r_1^{(\ell)}\}]=
\begin{pmatrix}
u_p^Tu_p & u_p^Tu_{r_1^{(\ell)}}\\
u_{r_1^{(\ell)}}^Tu_p & u_{r_1^{(\ell)}}^Tu_{r_1^{(\ell)}}
\end{pmatrix},
\]
whose determinant is nonzero. 
So the bilinear form on $L_\ell$  is nondegenerate.
Since $L_1,L_2$ are nondegenerate spaces of the same dimension over $\C$, they are isometric, 
and Witt's extension theorem~\cite[Chapter~I]{Lam} extends the identity on $\C u_p$ 
to an isometry $\Psi:L_2\to L_1$.
\end{proof}

Let $\Psi$ be the linear map shown in Claim~\ref{claim:leaf-path-isometry}. 
By Corollary~\ref{cor:5}, $\{p,q\}$ collapses in $X$, so $u_p,u_q$ are linearly dependent. 
Since $u_p\neq 0$ by \eqref{eq:leaf-cycle-edge-independence}, we have $u_q\in \C u_p$.
Hence, the linearity of $\Psi$ and $\Psi(u_p)=u_p$ imply $\Psi(u_q)=u_q$.
For each $i\in V(C)$, define
\[
v_i=
\begin{cases}
u_i,&i\in V(P_1),\\[2mm]
\Psi(u_i),&i\in V(P_2).
\end{cases}
\]
The two definitions agree at $p,q$. By Claim~\ref{claim:leaf-path-span}, all $v_i$ lie in $L_1$.
Since $\Psi$ preserves the bilinear form, we have
$
v_i^Tv_i=a_{ii}\ (i\in V(C))$ and $v_i^Tv_j=a_{ij}\ (\{i,j\}\in E(C))$.
Therefore $Y_C:=(v_i^Tv_j)_{i,j\in V(C)}$ is a completion of $p_{G,C}(a)$ with rank at most two. Its edge determinants agree with those of $X$ and are nonzero, so $\rank Y_C=2$.
\end{proof}

\subsection{Proof of the main theorem}
\label{sec:K4}

We now prove the main theorem of this paper, Theorem~\ref{thm:main}.
We first establish the clique--cycle decomposition for
$K_4$-minor-free graphs, and then extend it to cycle-completable
graphs using the clique-sum formula.
We begin with the following consequence of the clique--cycle
decomposition, which allows us to identify the essential
discriminants of the summands in the inductive proof.

\begin{prop}
\label{prop:14}
Let $G$ be a simple graph and suppose that
the clique--cycle decomposition holds for $G$.
Then $
\cR(G)=\bigcup_{C\in\operatorname{IndCyc}(G)}p_{G,C}^{-1}\cR(C)$.
\end{prop}

\begin{proof}
The assumption gives
$
\cD(G)\setminus\cKD(G)
\subset
\bigcup_{C\in\operatorname{IndCyc}(G)}p_{G,C}^{-1}\cR(C)$.
Since the right-hand side is closed, taking Zariski closures yields
$\cR(G)\subset
\bigcup_{C\in\operatorname{IndCyc}(G)}p_{G,C}^{-1}\cR(C)$.

For the reverse inclusion, fix $C\in\operatorname{IndCyc}(G)$.
By Proposition~\ref{prop:cycle-essential-hypersurface}, we see that $\cR(C)$ is irreducible, and hence so is $p_{G,C}^{-1}\cR(C)$. 
Moreover, $p_{G,C}^{-1}\cR(C)\subset\cD(G)$ by assumption. 
We claim that $p_{G,C}^{-1}\cR(C)\not\subset\cKD(G)$. 
Note that, for $a\in p_{G,C}^{-1}\cR(C)$, the diagonal and edge coordinates indexed by $C$ are constrained by the defining equation of $\cR(C)$, whereas all the other coordinates are free.
Let $L$ be a maximal clique of $G$. 
Since $C$ is chordless and has length at least $4$, the set $L\cap V(C)$ has at most two vertices, and two such vertices form an edge of $C$. 
If $L\cap V(C)=\varnothing$, choose $a[L]$ to be nonsingular.
If $L\cap V(C)=\{i\}$, 
choose \(j\in L\setminus V(C)\) and the free entries so that the \(\{i,j\}\)-principal block is \(\left(\begin{smallmatrix}a_{ii}&1\\1&0\end{smallmatrix}\right)\) and the remaining block is the identity.
Then $\det a[L]\neq 0$.
If $L\cap V(C)=\{i,j\}$, then $\{i,j\}\in E(C)$, and the definition of $\cR(C)$ allows us to choose $a$ with $\det a[\{i,j\}]\neq0$; choosing the remaining coordinates of $a[L]$ as an identity block again gives $\det a[L]\neq0$. Thus $p_{G,C}^{-1}\cR(C)$ is contained in none of the hypersurfaces defining $\cKD(G)$. Since $p_{G,C}^{-1}\cR(C)$ is irreducible, $p_{G,C}^{-1}\cR(C)\setminus\cKD(G)$ is Zariski dense in $p_{G,C}^{-1}\cR(C)$. Therefore
$p_{G,C}^{-1}\cR(C)\subset
\overline{\cD(G)\setminus\cKD(G)}^{\mathrm{Zar}}
=\cR(G)$.
Taking the union over all chordless cycles $C$ proves the reverse inclusion.
\end{proof}

We now turn to the clique--cycle decomposition for
$K_4$-minor-free graphs.
We first solve the case when $G$ is a $K_4$-minor-free
clique-sum atom, using the tools developed in the preceding
subsection.

\begin{prop}
\label{prop:atom-step}
Let $G$ be a $K_4$-minor-free clique-sum atom, and suppose that $G$ is neither $K_1$, $K_2$, $K_3$, nor a cycle. Assume that the clique--cycle decomposition holds for every proper induced subgraph of $G$. Then the same decomposition holds for $G$.
\end{prop}

\begin{proof}
By Lemma~\ref{lem:sp-tree-properties}(\ref{sp:triangle-free}), $G$ has no triangle. Proposition~\ref{prop:max-clique-component} and Theorem~\ref{thm:L} give 
$\cKD(G)\cup\bigcup_{C\in\operatorname{IndCyc}(G)}p_{G,C}^{-1}\cR(C)\subset\cD(G)$,
so it remains to prove the reverse inclusion. 
Fix $a\in\Sigma(G)\setminus\cKD(G)$, and choose a completion $X$ with $p_G(X)=a$ and a vertex-minimal nonzero $X$-stress $Z$. 
Let $S:=V(\operatorname{supp}Z)$.

First, suppose that $S\subsetneq V(G)$. Since $X[S]Z[S]=0$ and $Z[S]\neq O$, we have $p_{G[S]}(X[S])\in\cD(G[S])$. 
Here, we claim that $p_{G[S]}(X[S])\notin\cKD(G[S])$. Otherwise, there is a maximal clique $L$ of $G[S]$ with $\det X[L]=0$. Since $G$ has no triangle, $|L|\le2$. If $|L|=2$, then $L$ is also a maximal clique of $G$, and hence $a\in\cKD(G)$. 
This contradicts $a\notin\cKD(G)$. 
If $L=\{i\}$, then $i$ is isolated in $G[S]$. Since $i\in S$, the definition of $S$ forces $Z_{ii}\neq0$. 
Since $Z_{ii}\neq0$, the $i$th column of $XZ=O$ implies that the $i$th column of $X$ is zero. 
Since $G$ is a clique-sum atom, $i$ has a neighbor $j$ in $G$, and hence $\det X[\{i,j\}]=0$. 
Since $G$ has no triangle, the edge $\{i,j\}$ is a maximal clique of $G$, contradicting $a\notin\cKD(G)$. Thus $p_{G[S]}(X[S])\notin\cKD(G[S])$. 
By the assumption, there is a cycle $C\in\operatorname{IndCyc}(G[S])$ such that $p_{G,C}(a)\in\cR(C)$. Since $G[S]$ is an induced subgraph of $G$, we also have $C\in\operatorname{IndCyc}(G)$. 

Next, suppose that $S=V(G)$. Then $Z$ is a vertex-minimal full-support $X$-stress. By Lemma~\ref{prop:9}, there exists $C\in\operatorname{IndCyc}(G)$ such that $p_{G,C}(a)$ has a completion $Y_C$ of rank at most $2$. Since $a\notin\cKD(G)$ and $G$ has no triangle, every edge $e$ of $C$ satisfies $\det Y_C[e]=\det X[e]\neq0$. Hence $\rank Y_C=2$. Proposition~\ref{prop:cycle-rank2} gives $p_{G,C}(a)\in\Sigma(C)$. Since no cycle-edge determinant vanishes, $p_{G,C}(a)\notin\cKD(C)$. Thus $p_{G,C}(a)\in\cR(C)$ by the definition of $\cR(C)$.

In either case, we have
$
\Sigma(G)\setminus\cKD(G)
\subseteq\bigcup_{C\in\operatorname{IndCyc}(G)}p_{G,C}^{-1}\bigl(\cR(C)\bigr)$.
The right-hand side is Zariski closed, so taking Zariski closures gives
$
\cR(G)\subseteq\bigcup_{C\in\operatorname{IndCyc}(G)}p_{G,C}^{-1}\bigl(\cR(C)\bigr)$.
Since $\cD(G)=\cKD(G)\cup\cR(G)$, this proves the reverse inclusion.
\end{proof}

We now state the theorem for $K_4$-minor-free graphs.

\begin{thm}
\label{thm:15}
Let $G$ be a $K_4$-minor-free graph. Then
\[
\cD(G)=\cKD(G)\cup\bigcup_{C\in\operatorname{IndCyc}(G)}p_{G,C}^{-1}\cR(C).
\]
The defining polynomial of $\cD(G)$, up to a nonzero scalar, is
\[
\prod_{L\in\operatorname{MaxCliq}(G)}\det X[L]
\prod_{C\in\operatorname{IndCyc}(G)}\operatorname{ItrRes}(C).
\]
\end{thm}

\begin{proof}
First, we show the clique--cycle decomposition by induction on the number of vertices of $G$. If $G=K_1$, $K_2$, or $K_3$, then $\cD(G)=\cKD(G)=\cV_{A_G}(\det X[V(G)])$ obviously. If $G=C$ is a cycle of length at least $4$, Proposition~\ref{prop:max-clique-component} and the definition of the essential discriminant give $\cD(C)=\cKD(C)\cup\cR(C)$. 

Now assume that the result holds for every $K_4$-minor-free graph with fewer vertices than $G$. Suppose first that $G=G_1+_K G_2$ is a nontrivial clique-sum. Both $G_1$ and $G_2$ are $K_4$-minor-free and have fewer vertices than $G$. By the inductive hypothesis and Proposition~\ref{prop:14}, we have
$\cR(G_\ell)=\bigcup_{C\in\operatorname{IndCyc}(G_\ell)}p_{G_\ell,C}^{-1}\cR(C),
\ \ell=1,2$. 
By the clique-sum formula (Theorem~\ref{thm:CS}), we obtain
\[
\cD(G)=\cKD(G)
\cup\bigcup_{C\in\operatorname{IndCyc}(G_1)}p_{G,C}^{-1}\cR(C)
\cup\bigcup_{C\in\operatorname{IndCyc}(G_2)}p_{G,C}^{-1}\cR(C).
\]
Every chordless cycle of a clique-sum is contained in one of its summands, and hence $\operatorname{IndCyc}(G)=\operatorname{IndCyc}(G_1)\cup\operatorname{IndCyc}(G_2)$. This proves the clique--cycle decomposition for $G$ in the clique-sum case. If $G$ admits no nontrivial clique-sum decomposition, then $G$ is a clique-sum atom. Unless $G$ is one of the base cases above, Proposition~\ref{prop:atom-step} applies, since every proper induced subgraph of $G$ is $K_4$-minor free and has fewer vertices. This completes the induction.

Finally, each clique determinant $\det X[L]$ is irreducible, and
Corollary~\ref{cor:iterated-resultant-essential} shows that each
$\operatorname{ItrRes}(C)$ is irreducible and independent of the
triangulation and ear-removal order up to a nonzero scalar.
The factors indexed by different cliques or cycles depend on different
sets of variables and are therefore pairwise nonassociate.
Hence the clique--cycle decomposition implies that their product is
the defining polynomial of $\cD(G)$, up to a nonzero scalar.
\end{proof}


Finally, we prove Theorem~\ref{thm:main}, which extends
Theorem~\ref{thm:15} to cycle-completable graphs.

\begin{proof}[Proof of Theorem~\ref{thm:main}]
As defined in Subsection~\ref{subsec:def}, a cycle-completable
graph is obtained from complete graphs and $K_4$-minor-free
graphs by iterated clique-sums.
The clique--cycle decomposition holds for complete graphs,
since $\cD(H)=\cKD(H)$ and
$\operatorname{IndCyc}(H)=\varnothing$,
and for $K_4$-minor-free graphs by Theorem~\ref{thm:15}.

By Proposition~\ref{prop:14} and the clique-sum formula
(Theorem~\ref{thm:CS}), the decomposition is preserved under
clique-sums, since every chordless cycle of a clique-sum
is contained in one of its summands.
Thus the clique--cycle decomposition holds for $G$.
The defining polynomial follows as in the proof of
Theorem~\ref{thm:15}.
\end{proof}

\section*{Acknowledgments}
This work was supported by the Japan Science and Technology Agency (JST) as part of Adopting Sustainable Partnerships for Innovative Research Ecosystem (ASPIRE), Grant Number JPMJAP2520.
The first author was supported by JSPS KAKENHI Grant Number JP24K06841.

The authors used OpenAI ChatGPT (GPT-5.6 Thinking and GPT-6 Astra) to identify possible gaps in mathematical arguments, improve the English exposition, and suggest related literature. 
All suggestions and references were independently verified by the authors. The authors take full responsibility for the content of this paper.


\begin{thebibliography}{99}

\bibitem{AlizadehHaeberlyOverton}
F.~Alizadeh, J.-P.~Haeberly and M.~L.~Overton,
\textit{Complementarity and nondegeneracy in semidefinite programming},
Mathematical Programming \textbf{77} (1997), pp.111-128.

\bibitem{AlizadehHaeberlyOvertonIPM}
F.~Alizadeh, J.-P.~A.~Haeberly and M.~L.~Overton,
\textit{Primal-dual interior-point methods for semidefinite programming:
convergence rates, stability and numerical results},
SIAM J. Optim. \textbf{8} (1998), pp.~746--768.

\bibitem{BarrettJohnsonLoewy}
W.~W.~Barrett, C.~R.~Johnson and R.~Loewy,
\textit{The Real Positive Definite Completion Problem: Cycle Completability},
Memoirs of the American Mathematical Society \textbf{122}, no.~584,
American Mathematical Society, 1996.

\bibitem{BarrettJohnsonTarazaga}
W.~W.~Barrett, C.~R.~Johnson and P.~Tarazaga,
``The real positive definite completion problem for a simple cycle,''
\textit{Linear Algebra and its Applications} \textbf{192} (1993), 3--31.

\bibitem{BolkerRoth}
E.~D.~Bolker and B.~Roth,
``When is a bipartite graph a rigid framework?''
\textit{Pacific Journal of Mathematics} \textbf{90} (1980), no.~1, 27--44.


\bibitem{ChimaniHlineny}
M.~Chimani and P.~Hlin\v{e}n\'y,
``A tighter insertion-based approximation of the crossing number,''
\textit{Journal of Combinatorial Optimization} \textbf{33} (2017), no.~4, 1183--1225.


\bibitem{CLOIdeals}
D.~A.~Cox, J.~Little and D.~O'Shea,
\textit{Ideals, Varieties, and Algorithms:
An Introduction to Computational Algebraic Geometry and Commutative Algebra},
4th ed., Undergraduate Texts in Mathematics,
Springer, 2015.

\bibitem{CLOUsing}
D.~A.~Cox, J.~Little and D.~O'Shea,
\textit{Using Algebraic Geometry},
2nd ed., Graduate Texts in Mathematics 185,
Springer, 2005.

\bibitem{ChudnovskyMcInnis}
M.~Chudnovsky and I.~M.~J.~McInnis,
``Characterizing and generalizing cycle completable graphs,''
\textit{Discrete Mathematics} \textbf{347} (2024),
Article 113754.



\bibitem{FruehbisKruegerZach}
A.~Fr\"uhbis-Kr\"uger and M.~Zach,
``Determinantal singularities,''
in \textit{Handbook of Geometry and Topology of Singularities IV},
Springer, Cham, 2023, pp.~45--159.

\bibitem{GKZ}
I.~M.~Gelfand, M.~M.~Kapranov and A.~V.~Zelevinsky,
\textit{Discriminants, Resultants, and Multidimensional Determinants},
Birkh\"auser, Boston, 1994.

\bibitem{GroneJohnsonSaWolkowicz}
R.~Grone, C.~R.~Johnson, E.~M.~S\'a and H.~Wolkowicz,
``Positive definite completions of partial Hermitian matrices,''
\textit{Linear Algebra and its Applications} \textbf{58} (1984), 109--124.


\bibitem{vanDerHolstLovaszSchrijver}
H.~van der Holst, L.~Lov\'asz and A.~Schrijver,
``The Colin de Verdi\`ere graph parameter,''
in \textit{Graph Theory and Combinatorial Biology},
Bolyai Society Mathematical Studies \textbf{7},
J\'anos Bolyai Mathematical Society, Budapest, 1999, 29--85.

\bibitem{HornJohnson}
R.~A. Horn and C.~R. Johnson,
\textit{Matrix Analysis},
2nd ed.,
Cambridge University Press, Cambridge, 2013.

\bibitem{JacksonJordanTanigawa}
B.~Jackson, T.~Jord\'an and S.~Tanigawa,
``Combinatorial conditions for the unique completability of low-rank matrices,''
\textit{SIAM Journal on Discrete Mathematics} \textbf{28} (2014),
1797--1819.

\bibitem{JiangSturmfels}
Y.~Jiang and B.~Sturmfels,
\emph{Bad projections of the PSD cone},
Collect. Math. \textbf{72} (2021), 261--280.

\bibitem{JohnsonMcKee}
C.~R.~Johnson and T.~A.~McKee,
``Structural conditions for cycle completable graphs,''
\textit{Discrete Mathematics} \textbf{159} (1996), 155--160.



\bibitem{KriegeKurpiczMutzel}
N.~Kriege, F.~Kurpicz and P.~Mutzel,
``On maximum common subgraph problems in series--parallel graphs,''
\textit{European Journal of Combinatorics} \textbf{68} (2018), 79--95.

\bibitem{Lam}
T.~Y.~Lam,
\textit{Introduction to Quadratic Forms over Fields},
Graduate Studies in Mathematics 67,
American Mathematical Society, Providence, RI, 2005.

\bibitem{LaurentSeriesParallel}
M.~Laurent,
``The real positive semidefinite completion problem for series-parallel graphs,''
\textit{Linear Algebra and its Applications} \textbf{252} (1997), 347--366.



\bibitem{LaurentMascarinTelen}
M.~Laurent, F.~M.~Mascarin, and S.~Telen, ``The Algebraic Boundary of Graph Elliptopes,'' arXiv:2605.02484, 2026.

\bibitem{LaurentVarvitsiotisSAP}
M.~Laurent and A.~Varvitsiotis,
``Positive semidefinite matrix completion, universal rigidity and the Strong Arnold Property,''
\textit{Linear Algebra and its Applications} \textbf{452} (2014), 292--317.

\bibitem{MumfordRedBook}
D.~Mumford,
\textit{The Red Book of Varieties and Schemes},
Lecture Notes in Mathematics 1358,
Springer, Berlin, 1988.


\bibitem{QiChordalNondegeneracy} H.~Qi, Positive semidefinite matrix completions on chordal graphs and constraint nondegeneracy in semidefinite programming, \emph{Linear Algebra Appl.} \textbf{430} (2009), no.~4, 1151--1164.



\bibitem{SingerCucuringu}
A.~Singer and M.~Cucuringu,
``Uniqueness of low-rank matrix completion by rigidity theory,''
\textit{SIAM Journal on Matrix Analysis and Applications} \textbf{31} (2010),
no.~4, 1621--1641.

\bibitem{SUY}
L.~Solus, C.~Uhler and R.~Yoshida,
``Extremal positive semidefinite matrices whose sparsity pattern is given by graphs without $K_5$ minors,''
\textit{Linear Algebra and its Applications} \textbf{509} (2016), 247--275.

\bibitem{SturmfelsUhler}
B.~Sturmfels and C.~Uhler,
``Multivariate Gaussians, semidefinite matrix completion, and convex algebraic geometry,''
\textit{Annals of the Institute of Statistical Mathematics} \textbf{62} (2010), 603--638.






\end{thebibliography}
\end{document}